\documentclass[11pt]{article}
\usepackage{fullpage}
\usepackage{dsfont} 
\usepackage{multirow}
\usepackage{mathrsfs}
\usepackage{amsmath,amsthm}
\usepackage{mathtools}
\usepackage{hyperref}
\usepackage{cases}
\usepackage{amssymb}
\usepackage{graphicx}
\usepackage{pstricks}
\usepackage{pstricks-add}
\usepackage{epic}
\usepackage{color}
\usepackage{calc}
\usepackage{eucal} 
\usepackage{tikz, tikz-cd}
\usepackage{cite}
\usepackage{verbatim}
\usepackage[noblocks]{authblk}
\usepackage{enumitem}

\hypersetup{hidelinks} 
\theoremstyle{definition}
\newtheorem{definition}{Definition}[section]
\newtheorem{remark}[definition]{Remark}

\newtheorem{example}[definition]{Example}
\newtheorem{conjecture}[definition]{Conjecture}

\newtheorem{problem}[definition]{Problem}

\theoremstyle{plain}
\newtheorem{theorem}[definition]{Theorem}
\newtheorem{lemma}[definition]{Lemma}

\newtheorem{proposition}[definition]{Proposition}

\numberwithin{equation}{section}

\def\MX{\mathrm{Mat}_X(\mathbb{C})}
\def\<{\langle}
\def\>{\rangle}
\def\Spec{\mathrm{Spec}}
\def\1{\mathds{1}}

\begin{document}
\title{Remarks on pseudo-vertex-transitive graphs with small diameter}

\author[a,b]{Jack H. Koolen}
\author[c]{Jae-Ho Lee
	{\thanks{Corresponding author. \protect \\
	\indent \ \ 
	E-mail: koolen@ustc.edu.cn (J.H. Koolen), jaeho.lee@unf.edu (J.-H. Lee),  tansusan1@ahjzu.edu.cn (Y.-Y Tan)}}}
\author[d]{Ying-Ying Tan}

\affil[a]{\small School of Mathematical Sciences, University of Science and Technology of China, Hefei, Anhui, 230026, PR China.}
\affil[b]{CAS Wu Wen-Tsun Key Laboratory of Mathematics, University of Science and Technology of China, 96 Jinzhai Road, Hefei, Anhui, 230026, PR China.}
\affil[c]{Department of Mathematics and Statistics, University of North Florida, Jacksonville, FL 32224, USA.}
\affil[d]{School of Mathematics $\&$ Physics, Anhui Jianzhu University, Hefei, Anhui, 230201, PR China.}

\date{}

\maketitle

\begin{abstract}
Let $\Gamma$ denote a $Q$-polynomial distance-regular graph with vertex set $X$ and diameter $D$.
Let $A$ denote the adjacency matrix of $\Gamma$.
For a vertex $x\in X$ and for $0 \leq i \leq D$, let $E^*_i(x)$ denote the projection matrix to the $i$th subconstituent space of $\Gamma$ with respect to $x$.
The Terwilliger algebra $T(x)$ of $\Gamma$ with respect to $x$ is the semisimple subalgebra of $\mathrm{Mat}_X(\mathbb{C})$ generated by $A, E^*_0(x), E^*_1(x), \ldots, E^*_D(x)$.
Let $V$ denote a $\mathbb{C}$-vector space consisting of complex column vectors with rows indexed by $X$.
We say $\Gamma$ is pseudo-vertex-transitive whenever for any vertices $x,y \in X$, 
there exists a $\mathbb{C}$-vector space isomorphism $\rho:V\to V$ such that $(\rho A - A \rho)V=0$ and $(\rho E^*_i(x) - E^*_i(y)\rho)V=0$ for all $0\leq i \leq D$.
In this paper, we discuss pseudo-vertex transitivity for distance-regular graphs with diameter $D\in \{2,3,4\}$.
For $D=2$, we show that a strongly regular graph is pseudo-vertex-transitive if and only if all its local graphs have the same spectrum.
For $D = 3$, we consider the Taylor graphs and show that they are pseudo-vertex transitive.
For $D=4$, we consider the antipodal tight graphs and show that they are pseudo-vertex transitive.



\end{abstract}

\section{Introduction}
Let $\Gamma$ denote a $Q$-polynomial distance-regular graph with vertex set $X$ and diameter $D$.
We note that $\Gamma$ can be regarded as a combinatorial analogue of a compact rank one symmetric space \cite[p. 311]{bi}.
We recall the Terwilliger algebra of $\Gamma$.
Let $A$ denote the adjacency matrix of $\Gamma$.
Fix a \emph{base vertex} $x \in X$. 
For $0 \leq i \leq D$, let $E^*_i=E^*_i(x)$ denote the projection matrix onto the $i$th subconstituent space of $\Gamma$ with respect to $x$.
The \emph{Terwilliger algebra} (or \emph{subconstituent algebra}) $T(x)$ of $\Gamma$ with respect to $x$ is the subalgebra of $\MX$ generated by $A$, $E^*_0$, $E^*_1, \ldots, E^*_D$ \cite{ter1}.
Note that $T(x)$ is finite-dimensional and semisimple.
Let $T(y)$ denote the Terwilliger algebra of $\Gamma$ with respect to another base vertex $y\in X$.
We say the Terwilliger algebras $T(x)$ and $T(y)$ are \emph{isomorphic} whenever there exists a $\mathbb{C}$-algebra isomorphism from $T(x)$ to $T(y)$ that sends $A$ to $A$ and $E^*_i$ to $E^*_i(y)$ for $0 \leq i \leq D$.
We remark that the isomorphism class of $T(x)$ may depend on the choice of the base vertex $x$.
For instance, the twisted Grassmann graph, introduced by van Dam and Koolen \cite{2005vDK}, has two orbits of the automorphism group on its vertex set, say $X_1$ and $X_2$. 
For vertices $x\in X_1$ and $y\in X_2$, the Terwilliger algebras $T(x)$ and $T(y)$ of the twisted Grassmann graph are not isomorphic; cf \cite{bfk}. 
Let $V$ denote the $\mathbb{C}$-vector space consisting of complex column vectors with rows indexed by $X$.
Observe that $\MX$ acts on $V$ by left multiplication.
View $V$ as a left module for $T(x)$ and call this the \emph{standard module of $T(x)$} (or \emph{ standard $T(x)$-module}).
Since $T(x)$ is semisimple, $V$ decomposes into a direct sum of irreducible $T(x)$-modules.
We say $\Gamma$ is \emph{pseudo-vertex-transitive} whenever for every pair of vertices $x,y \in X$ there exists a $\mathbb{C}$-vector space isomorphism $\rho: V\to V$ such that $(\rho A - A \rho)V=0$ and $(\rho E^*_i - E^*_i(y)\rho)V=0$ for all $0\leq i \leq d$.
In other words, when $\Gamma$ is pseudo-vertex-transitive, it means that the isomorphism class of the standard $T(x)$-module of $\Gamma$ does not depend on the base vertex $x$; consequently, the multiplicities of irreducible $T(x)$-modules do not depend on the base vertex $x$.

We recall the notion of the thinness of $\Gamma$.
Let $W$ denote an irreducible $T(x)$-module.
Then $W$ is a direct sum of nonzero spaces among $E^*_0W, E^*_1W, \ldots, E^*_DW$, and also a direct sum of nonzero spaces among $E_0W, E_1W, \ldots, E_DW$, where $E_i$ is the $i$th primitive idempotent of $\Gamma$.
We note that the dimension of $E^*_iW$ is at most $1$ for $0 \leq i \leq D$ if and only if the dimension of $E_iW$ is at most $1$ for $0 \leq i \leq D$ \cite[Lemma 3.9]{ter1}; in this case, $W$ is called \emph{thin}.
The graph $\Gamma$ is called \emph{thin} whenever every irreducible $T(x)$-module is thin for every vertex $x\in X$.

In the present paper, we discuss the thinness and pseudo-vertex transitivity of $Q$-polynomial distance-regular graphs with small diameter $D\in \{2,3,4\}$.
For $D=2$, a distance-regular graph is strongly regular. 
We show that pseudo-vertex transitivity of a strongly regular graph is determined by the spectrum of its local graph; cf. Theorem \ref{prop:char_p.v.t}.
For $D\in \{3,4\}$, we discuss only the distance-regular antipodal double covers. 
Let $\Gamma$ denote a distance-regular antipodal double cover with diameter $D$.
If $D=3$, then $\Gamma$ is a Taylor graph. 
We discuss the thinness of $\Gamma$ and show that the isomorphism class of the standard module of the Terwilliger algebra of $\Gamma$ is determined by its intersection numbers; thus, it is pseudo-vertex-transitive; cf. Theorem \ref{thm:D=3}.
If $D=4$, $\Gamma$ is $Q$-polynomial if and only if $\Gamma$ is tight \cite{2002JKDiscMath}. 
We discuss the thinness of $\Gamma$ and show that $\Gamma$ is pseudo-vertex-transitive, provided that $\Gamma$ is tight; cf. Theorem \ref{thm:D=4}.
We note that bipartite and/or antipodal $Q$-polynomial distance-regular graphs are pseudo-vertex-transitive; cf. Proposition \ref{prop:antiDRG pvt}.
In a future paper, we will discuss pseudo-vertex transitivity of general $Q$-polynomial distance-regular graphs with diameter $D\geq 3$.

We remark that the thin property of an antipodal $Q$-polynomial double cover with $D\in \{3,4\}$ plays an important role in determining pseudo-vertex transitivity. 
From this point of view, we could ask the following question: for $D\geq 3$ does the thinness of a $Q$-polynomial distance-regular graph guarantee its pseudo-vertex transitivity?
Recently, Ito and Koolen suggested the following problem:
\begin{problem}[The Ito-Koolen Problem \cite{2020TKCP}]\label{IK problem} 
Classify all thin pseudo-vertex-transitive $Q$-polynomial distance-regular graphs with large enough diameter.
\end{problem}
\noindent
The ultimate goal of our study is to classify thin pseudo-vertex-transitive $Q$-polynomial distance-regular graphs.
The present paper shows that such graphs with small diameter are difficult to classify.

This paper is organized as follows.
Sections \ref{Sec:Prelim} and \ref{sec:Ter_Alg} review some preliminaries concerning distance-regular graphs and the Terwilliger algebra.
Section \ref{Sec:PVT} introduces the notion of a pseudo-vertex-transitive distance-regular graph.
It also contains what it means for two $Q$-polynomial distance-regular graphs to be $T$-isomorphic.
Section \ref{sec:SRG} deals with the thinness of strongly regular graphs and contains a characterization of their pseudo-vertex transitivity. 
It also contains some examples of cospectral strong regular graphs and discusses their pseudo-vertex transitivity.
Section \ref{sec:antipodal 2-cover} reviews some preliminaries concerning the antipodal distance-regular graphs.
Section \ref{sec:D=3} deals with Taylor graphs and their thinness and pseudo-vertex transitivity. 
Section \ref{sec:D=4} discusses the antipodal tight graphs with $D=4$.
The paper ends with a brief summary and directions for future work in Section \ref{Sec:conclusion}.

Throughout the paper, we use the following notation.
Let $X$ denote a nonempty finite set.
Let $\MX$ denote the $\mathbb{C}$-algebra consisting of all complex matrices whose rows and columns are indexed by $X$.
Let $V=\mathbb{C}^X$ denote the $\mathbb{C}$-vector space consisting of all complex column vectors with rows indexed by $X$.
We endow $V$ with the standard Hermitian inner product $\< u,v \>=u^t\overline{v}$ for $u, v \in V$.
We view $V$ as a left module for $\MX$, called the \emph{standard module}.

\section{Preliminaries: distance-regular graphs}\label{Sec:Prelim}
In this section, we recall some definitions and notation.
Let $\Gamma$ denote a finite, simple, undirected, connected graph with vertex set $X$ and diameter $D=\max\{\partial(x,y)\mid x,y \in X\}$, where $\partial$ denotes the path-length distance function for $\Gamma$.
For a vertex $x \in X$, define $\Gamma_i(x) = \{y\in X \mid \partial(x,y)=i\}$ for $0 \leq i \leq D$.
The subgraph $\Delta_i(x)$, $0\leq i \leq D$, of $\Gamma$ induced by $\Gamma_i(x)$ is called the $i$th \emph{subconstituent} of $\Gamma$ with respect to $x$. 
The number $|\Gamma_1(x)|$ is called the \emph{valency} of $x$ in $\Gamma$.
A graph $\Gamma$ is said to be \emph{regular with valency} $k$ (or \emph{$k$-regular}) if each vertex of $\Gamma$ has valency $k$.
We abbreviate $\Delta(x)=\Delta_1(x)$, the first subconstituent of $\Gamma$ with respect to $x$, and call this the \emph{local} graph of $\Gamma$ at $x$.
Let $\lambda$ denote an eigenvalue of $\Delta(x)$. 
We call $\lambda$ a \emph{local eigenvalue} of $\Gamma$ with respect to $x$.
We say a graph $\Gamma$ is \emph{locally} $\Delta$ whenever all local graphs of $\Gamma$ are isomorphic to $\Delta$.

We say $\Gamma$ is \emph{distance-regular} whenever for all integers $0\leq h,i,j \leq D$ and for all vertices $x,y \in X$ with $\partial(x,y)=h$, the number $p^h_{ij} = |\Gamma_i(x)\cap \Gamma_j(y)|$ is independent of $x$ and $y$.
The constants $p^h_{ij}$ are called the \emph{intersection numbers} of $\Gamma$.
We abbreviate $a_i=p^i_{1,i} (0 \leq i \leq D)$, $b_i=p^i_{1,i+1} (0 \leq i \leq D-1)$ and $c_i=p^i_{1, i-1} (1\leq i \leq D)$.
Observe $\Gamma$ is regular with valency $k=b_0$, and $c_i+a_i+b_i=k$ for $0 \leq i \leq D$, where we define $c_0=0$ and $b_D=0$.
The sequence $\{b_0, b_1,\ldots, b_{D-1}; c_1, c_2, \ldots, c_D\}$ is called the \emph{intersection array} of $\Gamma$.
We say $\Gamma$ is \emph{bipartite} whenever $a_i=0$ for $0\leq i \leq D$.

We assume that $\Gamma$ is distance-regular with diameter $D$. 
For $0 \leq i \leq D$, let $A_i$ denote the matrix in $\MX$ defined by 
\begin{equation*}
	(A_i)_{xy} = \begin{cases}
	1 & \quad \text{if} \quad \partial(x,y)=i,\\
	0 & \quad \text{if} \quad \partial(x,y)\ne i,\\
	\end{cases}
	\qquad (x,y  \in X).
\end{equation*}
We call $A_i$ the $i$th \emph{distance matrix} of $\Gamma$.
We abbreviate $A: =A_1$, called the \emph{adjacency matrix} of $\Gamma$.
Observe (i) each $A_i$ is real symmetric; (ii) $A_0=I$; (iii) $\sum^D_{i=0}A_i=J$, the all-ones matrix; (iv) $A_iA_j = \sum^D_{h=0}p^h_{ij}A_h$ $(0\leq i,j \leq D)$.
By these facts, we find that $A_0, A_1,\ldots, A_D$ is a basis for a commutative subalgebra $M$ of $\MX$, which we call the \emph{Bose-Mesner algebra} of $\Gamma$.
It is known that $A$ generates $M$.
The algebra $M$ has a second basis $E_0, E_1, \ldots, E_D$ such that (i) $E_0=|X|^{-1}J$; (ii) $\sum^D_{i=0}E_i=I$; (iii) $E_iE_j=\delta_{ij}E_i$; cf. \cite[p.45]{1989BCN}. 
We call $E_i$ the \emph{$i$th primitive idempotent} of $\Gamma$.
Since $\{E_i\}^D_{i=0}$ is a basis for $M$, there exist complex scalars $\{\theta_i\}^D_{i=0}$ such that $A=\sum^D_{i=0}\theta_i E_i$.
Observe $AE_i=E_iA=\theta_iE_i$ for $0\leq i \leq D$. 
The scalars $\{\theta_i\}^D_{i=0}$ are real \cite[p.197]{bi}, and mutually distinct as $A$ generates $M$. 
We call $\theta_i$ the \emph{eigenvalue} of $\Gamma$ associated with $E_i$ for $0\leq i \leq D$.
Observe $V=E_0V+E_1V+\cdots + E_DV$, an orthogonal direct sum.
For $0\leq i \leq D$, $E_iV$ is the eigenspace of $A$ associated with $\theta_i$. 
For $0\leq i \leq D$, we denote by $m_i$ the rank of $E_i$ and observe $m_i=\dim(E_iV)$.
We call $m_i$ the \emph{multiplicity} of $\theta_i$.
By the \emph{spectrum} of $\Gamma$, we mean the multiset containing its eigenvalues, each with its multiplicity, denoted by $\Spec(\Gamma)=\{\theta^{m_0}_0, \theta^{m_1}_1, \ldots, \theta^{m_D}_D\}$.

We recall the notion of the $Q$-polynomial property of $\Gamma$.
Let $\circ$ denote the entrywise product in $\MX$.
Since $A_i\circ A_j =\delta_{ij}A_i$ $(0 \leq i, j \leq D)$, the Bose-Mesner algebra $M$ is closed under $\circ$.
Since $\{E_i\}^D_{i=0}$ is a basis for $M$, there exist complex scalars $q^h_{ij}$ such that 
\begin{equation*}
E_i\circ E_j = |X|^{-1}\sum^D_{h=0}q^h_{ij}E_h, \quad (0 \leq i,j \leq D).
\end{equation*}
By \cite[p.48, 49]{1989BCN} the scalars $q^h_{ij}$ are real and nonnegative.
The $q^h_{ij}$ are called the \emph{Krein parameters} of $\Gamma$.
We say $\Gamma$ is \emph{$Q$-polynomial} (with respect to the ordering $E_0, E_1, \ldots, E_D$) whenever for all integers $0\leq h,i,j\leq D$, $q^h_{ij}=0$ (resp. $q^h_{ij}\ne0$) if one of $h,i,j$ is greater than (resp. equal to) the sum of the other two \cite[p.235]{1989BCN}.
{
From now on, unless otherwise stated, we assume that a $Q$-polynomial distance-regular graph discussed in this paper has the ordering $E_0, E_1, \ldots, E_D$.
}
For more background information about distance-regular graphs, we refer the reader to \cite{1989BCN, bi, dkt}.

\section{Preliminaries: the Terwilliger algebra}\label{sec:Ter_Alg}
In this section, we recall the Terwilliger algebra of a $Q$-polynomial distance-regular graph.
Let $\Gamma$ denote a $Q$-polynomial distance-regular graph with vertex set $X$ and diameter $D$. 
Fix a vertex $x \in X$. We refer to $x$ as a ``base'' vertex.
For $0 \leq i \leq D$, we define the diagonal matrix $E^*_i=E^*_i(x) \in \MX$ with diagonal entry
\begin{equation*}
	(E^*_i)_{yy} = \begin{cases}
	1 & \quad \text{if} \quad \partial(x,y)=i,\\
	0 & \quad \text{if} \quad \partial(x,y)\ne i,\\
	\end{cases}
	\qquad (y  \in X).
\end{equation*}
We call $E^*_i$ the $i$th \emph{dual primitive idempotent} of $\Gamma$ with respect to $x$.
Observe $\sum^D_{i=0}E^*_i=I$ and $E^*_iE^*_j=\delta_{ij}E^*_i$.
By these facts, $E^*_0, E^*_1, \ldots, E^*_D$ is a basis for a commutative subalgebra $M^*=M^*(x)$ of $\MX$ with respect to $x$, which we call the \emph{dual Bose-Mesner algebra} of $\Gamma$.
Recall the primitive idempotents $\{E_i\}^D_{i=0}$ of $\Gamma$.
For $0 \leq i \leq D$, define the diagonal matrix $A^*_i=A^*_i(x) \in \MX$ with diagonal entry $(A^*_i)_{yy} = |X|(E_i)_{xy}$ for $y\in X$.
By \cite[p.379]{ter1}, $A^*_0, A^*_1, \ldots, A^*_D$ is a basis for $M^*$. 
Moreover $A^*_0=I$ and $A^*_iA^*_j=\sum^D_{h=0} q^h_{ij} A^*_h$.
We call $A^*_i$ the $i$th \emph{dual distance matrix} of $\Gamma$ with respect to $x$.
We abbreviate $A^*=A^*_1$ and call this the \emph{dual adjacency matrix} of $\Gamma$ with respect to $x$.
The matrix $A^*$ generates $M^*$ \cite[Lemma 3.11]{ter1}.
Since $\{E^*_i\}^D_{i=0}$ is a basis for $M^*$, there exist complex scalars $\{\theta^*_i\}^D_{i=0}$ such that $A^*=\sum^D_{i=0}\theta^*_iE^*_i$.
Observe $A^*E^*_i=E^*_iA^*=\theta^*_iE^*_i$ for $0 \leq i \leq D$.
The scalars $\{\theta^*_i\}^D_{i=0}$ are real \cite[Lemma 3.11]{ter1} and mutually distinct.
We call $\theta^*_i$ the \emph{dual eigenvalue} of $\Gamma$ associated with $E^*_i$.
Observe $V=E^*_0V+E^*_1V+\cdots +E^*_DV$, an orthogonal direct sum.
For $0 \leq i \leq D$, $E^*_iV$ is the eigenspace of $A^*$ associated with $\theta^*_i$, which we call the $i$th \emph{subconstituent space} of $\Gamma$ with respect to $x$.

Recall the Bose-Mesner algebra $M$ of $\Gamma$ and the dual Bose-Mesner algebra  $M^*(x)$ of $\Gamma$ with respect to $x$.
Let $T=T(x)$ denote the subalgebra of $\MX$ generated by $M$ and $M^*$.
We call $T$ the \emph{Terwilliger algebra} (or \emph{subconstituent algebra}) of $\Gamma$ with respect to $x$ \cite{ter1}.
Note that $A$ and $A^*$ generate $T$.
The algebra $T$ is finite-dimensional and noncommutative.
The algebra $T$ is semisimple since it is closed under the conjugate-transpose map.
By \cite[Lemma 3.2]{ter1}, we have the following relations in $T$. 
For $0 \leq i,j \leq D$,
\begin{align*}
	E^*_iAE^*_j = 0 \quad \text{if } \quad |i-j|>1,\\
	E_iA^*E_j = 0 \quad \text{if } \quad |i-j|>1.
\end{align*} 

By a $T$-module, we mean a subspace $W$ of $V$ such that $BW\subseteq W$ for all $B\in T$.
Observe that $V$ is a $T$-module, called the \emph{standard module of $T$} (or \emph{standard $T$-module}).
A $T$-module $W$ is called \emph{irreducible} if $W\ne 0$ and $W$ contains no $T$-modules other than $0$ and $W$.
Two $T$-modules $W$, $W'$ are \emph{isomorphic} if there exists a $\mathbb{C}$-vector space isomorphism $\sigma:W \to W'$ such that
\begin{equation*}
	(\sigma B - B\sigma)W=0,
\end{equation*}
for all $B \in T$.
Let $W$ be a $T$-module and let $U$ be a $T$-submodule of $W$. 
The orthogonal complement of $U$ in $W$ is a $T$-module since $T$ is closed under the conjugate transpose map.
It follows that $W$ decomposes into an orthogonal direct sum of irreducible $T$-modules.
We observe that $V$ decomposes an orthogonal direct sum of irreducible $T$-modules; choose one of the irreducible $T$-modules in this decomposition of $V$, denoted by $W$.
By the \emph{multiplicity} of $W$, we mean the number of irreducible $T$-modules in this decomposition which are isomorphic to $W$ as $T$-modules.

Let $W$ be an irreducible $T$-module.
Then $W$ decomposes into a direct sum of nonzero spaces among $E^*_0W, E^*_1W, \ldots, E^*_DW$, and also a direct sum of nonzero spaces among $E_0W, E_1W, \ldots, E_DW$.
By the \emph{endpoint} of $W$, we mean $\min\{i \mid 0 \leq i \leq D, E^*_iW\ne 0\}$.
By the \emph{dual endpoint} of $W$, we mean $\min\{i \mid 0 \leq i \leq D, E_iW\ne 0\}$.
By the \emph{diameter} of $W$, we mean $|\{i \mid 0 \leq i \leq D, E^*_iW\ne 0\}|-1$.
Let $r$ denote the endpoint of $W$ and $d$ the diameter of $W$.
By \cite[Lemma 3.9]{ter1},  $E^*_iW \ne 0$  if and only if $r \leq i \leq r+d$, and $W=\sum^d_{h=0}E^*_{r+h}W$, an orthogonal direct sum; in addition, $\dim(E^*_iW)\leq 1$ for $0 \leq i \leq D$ if and only if $\dim(E_iW)\leq 1$ for $0 \leq i \leq D$.
An irreducible $T$-module $W$ is said to be \emph{thin} whenever $\dim(E^*_iW)\leq 1$ for $0 \leq i \leq D$.
There exists a unique thin irreducible $T$-module with endpoint $0$ and diameter $D$, which we call the \emph{primary} $T$-module.
Let $W$ denote the primary $T$-module. 
Let $v_0, v_1, \ldots, v_D$ be a sequence of vectors of $W$, not all zero.
This sequence is said to be a \emph{standard basis} for $W$ whenever both (i) $v_i \in E^*_iW$ for $0 \leq i \leq D$; and (ii) $\sum^D_{i=0} v_i \in E_0W$.
For instance, the sequence $E^*_0\mathds{1}, E^*_1\mathds{1}, \ldots ,E^*_D\mathds{1}$ is a standard basis for $W$\cite[Lemma 3.6]{ter1}, where $\mathds{1}$ is the all-ones vector.
We note that the action of $A$ on a standard basis for $W$ is given by
\begin{equation}\label{action: 3-term}
	Av_i = b_{i-1}v_{i-1} + a_iv_i + c_{i+1}v_{i+1} \qquad (0 \leq i \leq D),
\end{equation}
where $a_i, b_i, c_i$ are the intersection numbers of $\Gamma$ and $b_{-1}v_{-1}:=0$ and $c_{D+1}v_{D+1}:=0$.

Let $W$ be a thin irreducible $T$-module with endpoint $r$ and diameter $d$. 
We recall the actions of $A$ and $A^*$ on $W$.
Take a nonzero vector $v_0 \in E^*_rW$. 
For $1 \leq i \leq d$, define $v_i = E^*_{r+i} A v_{i-1}$ in $E^*_{r+i}W$.
Observe that the vector $v_i$ is a basis for $E^*_{r+i}W$ for each $0 \leq i \leq d$, and thus $v_0, v_1, \ldots, v_d$ is an orthogonal basis for $W$.
Define the scalars $a_i(W), 0 \leq i \leq d,$ and $x_i(W), 1\leq i \leq d$, by
\begin{equation}\label{scalars:ai,wi}
	a_i(W) = \mathrm{trace}(E^*_{r+i}A |_{W} ), \qquad x_i(W) = \mathrm{trace}(E^*_{r+i}AE^*_{r+i-1}A |_{W}),
\end{equation}
where $B|_{W}$ denotes the restriction of $B$ to $W$.
Let $\theta_i$ be the eigenvalue of $\Gamma$ associated with $E_i$ for $0 \leq i \leq D$.
Note that (cf. \cite[Lemma 5.10]{2010Cerzo})
\begin{equation}\label{sum:ai=thi}
	\sum^d_{i=0} a_i(W) = \sum^d_{i=0} \theta_{t+i},
\end{equation}
where $t$ is the dual endpoint of $W$.
For $0 \leq i \leq d$, the action of $A$ on $v_i$ is given as
\begin{equation}\label{eq:action Av_i}
	Av_i = v_{i+1} + a_i(W)v_i + x_i(W)v_{i-1},
\end{equation}
where $v_{d+1}:=0$ and $x_0(W)v_{-1}:=0$; cf. \cite[Theorem 5.7]{2010Cerzo}.
The action of $A^*$ on $v_i$ is given as  $A^*v_i = \theta^*_{r+i} v_i$,
where $\theta^*_j$ is the dual eigenvalue of $\Gamma$ associated with $E^*_j$.

The graph $\Gamma$ is said to be \emph{thin with respect to $x$} whenever every irreducible $T(x)$-module is thin.
The graph $\Gamma$ is said to be \emph{thin} whenever $\Gamma$ is thin with respect to every vertex $x$ of $\Gamma$.
See \cite[Section 6]{ter3} for examples of thin $Q$-polynomial distance-regular graphs.

Recall the local graph $\Delta(x)$ of $\Gamma$ at $x\in X$.
Observe that $\Delta(x)$ has $b_0=k$ vertices and is $a_1$-regular.
Let $a_1=\lambda_1\geq \lambda_2\geq \cdots \geq \lambda_k$ denote the local eigenvalues of $\Gamma$ with respect to $x$.
The spectrum of $\Delta(x)$ is called the \emph{local spectrum} of $\Gamma$ at $x$.
Let $W$ denote a thin irreducible $T$-module with endpoint $1$.
We observe that $E^*_1W$ is a one-dimensional eigenspace for $E^*_1AE^*_1$ with corresponding eigenvalue, say $\lambda$. 
Note that $\lambda$ is one of $\lambda_2, \ldots, \lambda_k$.
We call $\lambda$ the \emph{local eigenvalue} of $W$.

We finish this section with a comment, which will be useful later when calculating the dimension of the Terwilliger algebra.
\begin{proposition}\label{Prop:Wedderburn}
Let $T=T(x)$ be the Terwilliger algebra of $\Gamma$ with respect to $x$ and $V$ the standard $T$-module.
Suppose that $V$ has exactly $r$ non-isomorphic irreducible $T$-modules $W_{11}, W_{21}, \ldots, W_{r1}$ with $\dim(W_{i1})=n_i$.
For $1\leq i \leq r$, let $m_i$ be the multiplicity of $W_{i1}$.
Then $V$ decomposes into a direct sum of irreducible $T$-modules:
$$
	V = \bigoplus^{r}_{i=1} (\bigoplus^{m_i}_{j=1}W_{ij}),
$$
where $W_{ij}$ and $W_{i'j'}$ are isomorphic as $T$-modules if and only if $i=i'$.
The algebra $T$ is isomorphic to the semisimple algebra
$$
	\mathrm{Mat}_{n_1}(\mathbb{C}) \oplus \cdots \oplus \mathrm{Mat}_{n_r}(\mathbb{C})
$$
and 
\begin{equation}\label{T:dim formula}
	\dim(T) = \sum^r_{i=1}n^2_i.
\end{equation}
\end{proposition}
\begin{proof}
Follows from Wedderburn's theory \cite{1962CurtisReiner}.
\end{proof}

\section{Pseudo-vertex transitivity}\label{Sec:PVT}
In this section, we recall the notion of a pseudo-vertex-transitive graph.
Throughout this section, we denote by $\Gamma$ a $Q$-polynomial distance-regular graph with vertex set $X$ and diameter $D$.
Let $A$ denote the adjacency matrix of $\Gamma$.
For $x \in X$, let $A^*(x)$ denote the dual adjacency matrix of $\Gamma$ with respect to $x$ and let $T(x)$ denote the Terwilliger algebra of $\Gamma$ with respect to $x$.
Recall the standard module $V=\mathbb{C}^X$.
We now give the following definition.

\begin{definition}\label{pvt}
The graph $\Gamma$ is said to be \emph{pseudo-vertex-transitive} whenever for every pair of vertices $x,y \in X$, there exists a $\mathbb{C}$-vector space isomorphism $\rho: V\to V$ such that 
\begin{equation}\label{pvt:eq}
	(\rho A-A \rho)V=0, \qquad (\rho A^*(x)-A^*(y) \rho)V=0.
\end{equation}
Suppose that for any two vertices $x,y \in X$ there exists a $\mathbb{C}$-vector space isomorphism $\rho: V\to V$ satisfying \eqref{pvt:eq}.
Then we say that the standard $T(x)$-module and the standard $T(y)$-module are \emph{isomorphic}.
\end{definition}

\begin{lemma}\label{lem:pvt [A]_B1=[A]_B2}
$\Gamma$ is pseudo-vertex-transitive if and only if for every pair of vertices $x,y \in X$ there exist ordered bases $\mathcal{B}_x$ and $\mathcal{B}_y$ for $V$ such that the matrix representing $A$ (resp. $A^*(x)$) with respect to $\mathcal{B}_x$ is equal to the matrix representing $A$ (resp. $A^*(y)$) with respect to $\mathcal{B}_y$.
\end{lemma}
\begin{proof}
Immediate from Definition \ref{pvt}.
\end{proof}

For any two vertices $x, y \in X$, consider the Terwilliger algebras $T(x)$ and $T(y)$ of $\Gamma$.
We say $T(x)$ and $T(y)$ are \emph{isomorphic} whenever there exists a $\mathbb{C}$-algebra isomorphism from $T(x)$ to $T(y)$ that sends $A \mapsto A$ and ${A^*}(x) \mapsto {A^*}(y)$.
By Lemma \ref{lem:pvt [A]_B1=[A]_B2}, $T(x)$ and $T(y)$ are isomorphic, provided that $\Gamma$ is pseudo-vertex-transitive.
That is, the isomorphism class of the Terwilliger algebra of a pseudo-vertex-transitive $\Gamma$ does not depend on the base vertex.

\begin{remark}
(i) Let $\mathrm{Aut}(\Gamma)$ denote the automorphism group of $\Gamma$.
If $\mathrm{Aut}(\Gamma)$ has a single group orbit on $X$, then $\Gamma$ is pseudo-vertex-transitive. 
From this, it follows that every vertex-transitive graph is pseudo-vertex-transitive.
The converse of this statement is not true; see Example \ref{ex:GQ} and Remark \ref{rmk:Taylor graphs}.\\
(ii) The twisted Grassmann graph $\tilde J_q(2D+1,D)$ is not vertex-transitive; it has two orbits of the automorphism group on its vertex set, say $X_1$ and $X_2$.
In \cite{bfk}, Bang \textit{et al.} showed that all irreducible modules for the Terwilliger algebra $T(x)$ of $\tilde J_q(2D+1,D)$ with respect to $x \in X_2$ are thin, and there are non-thin irreducible modules for the Terwilliger algebra $T(y)$ of $\tilde J_q(2D+1,D)$ with respect to $y \in X_1$.
It implies that $\tilde J_q(2D+1,D)$ is not pseudo-vertex-transitive.
We remark that Tanaka and Wang determined all irreducible $T(x)$-modules of $\tilde J_q(2D+1,D)$ for the thin case; cf \cite{2020TanakaWang}. \\
(iii) With reference to (ii), let $x \in X_2$ be a base vertex of $\tilde J_q(2D+1,D)$ and let $T(x)$ be the Terwilliger algebra of $\tilde J_q(2D+1,D)$. 
We recall the Grassmann graph $J_q(2D+1,D)$.
We note that $J_q(2D+1,D)$ is thin \cite{ter3}.
For a base vertex $y$ of $J_q(2D+1,D)$, let $T(y)$ be the Terwilliger algebra of $J_q(2D+1,D)$. 
The local eigenvalues of $\tilde J_q(2D+1,D)$ at $x$ equal to the local eigenvalues of $J_q(2D+1,D)$ at $y$. From this, it follows that the isomorphism classes of the irreducible $T(x)$-modules with endpoint $1$ are the same as the isomorphism classes of the irreducible $T(y)$-modules with endpoint 1. However, their multiplicities are different; cf. \cite{bfk}.
In addition, it turns out that the Terwilliger algebras $T(x)$ and $T(y)$ are not isomorphic to each other; cf. \cite{2020TanakaWang}. 
\end{remark}

Suppose that $\Gamma$ is pseudo-vertex-transitive.
For $x \in X$ and for $0 \leq i \leq D$, let $\Delta_i(x)$ denote the $i$th subconstituent of $\Gamma$ with respect to $x$.
Then the spectrum of $\Delta_i(x)$ is determined by the characteristic polynomial for $E^*_i(x)AE^*_i(x)$.
Let $y$ be another vertex in $X$.
Since $\Gamma$ is pseudo-vertex-transitive, by Lemma \ref{lem:pvt [A]_B1=[A]_B2} there are ordered bases $\mathcal{B}_x$ and $\mathcal{B}_y$ for $V$ such that the matrix representing $E^*_i(x)AE^*_i(x)$ with respect to $\mathcal{B}_x$ is equal to the matrix representing $E^*_i(y)AE^*_i(y)$ with respect to $\mathcal{B}_y$. 
It follows that the characteristic polynomial for $E^*_i(x)AE^*_i(x)$ is equal to the characteristic polynomial for $E^*_i(y)AE^*_i(y)$. 
Therefore we obtain the following lemma.
\begin{lemma}\label{pvt=>Spec(D1)=Spec(D2)}
Suppose that $\Gamma$ is pseudo-vertex-transitive. 
Then for $0 \leq i \leq D$ and for all vertices $x,y \in X$, the spectrum of $\Delta_i(x)$ and the spectrum of $\Delta_i(y)$ are equal to each other.
\end{lemma}

Next, we generalize the concept of pseudo-vertex transitivity.
Let $\Gamma'$ denote a $Q$-polynomial distance-regular graph with vertex set $X'$ and diameter $D'$.
Fix a vertex $x' \in X'$.
Denote by $A'$ the adjacency matrix of $\Gamma'$ and ${A^*}'(x')$ the dual adjacency matrix of $\Gamma'$ with respect to $x'$.
Let $T'(x')$ be the Terwilliger algebra of $\Gamma'$ with respect to $x'$.
Then $T(x)$ and $T'(x')$ are said to be \emph{$T$-algebra isomorphic} whenever $D=D'$ and $|X|=|X'|$ and there exists a $\mathbb{C}$-algebra isomorphism from $T(x)$ to $T'(x')$ that sends $A \mapsto A'$ and ${A^*}(x) \mapsto {A^*}'(x')$.
In this case, such an algebra isomorphism is called a \emph{$T$-algebra isomorphism}.
Let $W$ denote a $T(x)$-module of $\Gamma$ and let $W'$ denote a $T'(x')$-module of $\Gamma'$.
Then $W$ and $W'$ are said to be \emph{$T$-isomorphic} whenever (i) there exists a $T$-algebra isomorphism $f:T(x) \to T'(x')$; and (ii) there exists a $\mathbb{C}$-vector space isomorphism $\rho:W \to W'$ such that $(\rho A - A \rho)W=0$ and $(\rho A^*(x) - {A^*}'(x') \rho)W=0$.
Recall $V=\mathbb{C}^X$ the standard $T(x)$-module of $\Gamma$.
Let $V'=\mathbb{C}^{X'}$ denote the standard $T'(x')$-module of $\Gamma'$.
With the above notation, we give the following definition.

\begin{definition}\label{T-isomorphic}
Two $Q$-polynomial distance-regular graphs $\Gamma$ and $\Gamma'$ are said to be \emph{$T$-isomorphic} whenever for every pair of vertices $x\in X$ and $x' \in X'$, 
(i) there exists a $T$-algebra isomorphism $f:T(x) \to T'(x')$; and 
(ii) there exists a $\mathbb{C}$-vector space isomorphism $\rho:V \to V'$ such that 
$(\rho B - f(B) \rho)V=0$
for all $B\in T(x)$; that is, the diagram
\begin{equation*}
	\begin{tikzcd}
	V \arrow{r}{\rho} \arrow[swap]{d}{B} & V' \arrow{d}{f(B)} \\
	V \arrow{r}{\rho} & V'
	\end{tikzcd}
\end{equation*}
commutes for all $B\in T(x)$.
\end{definition}

\begin{lemma}\label{lem:[A]_B=[A']_B'}
With the above notation, the following (i)-(iii) are equivalent.
\begin{itemize}[itemsep=0pt]
	\item[(i)] $\Gamma$ and $\Gamma'$ are $T$-isomorphic.
	\item[(ii)] For any vertices $x\in X$ and $x' \in X'$, there exists a $\mathbb{C}$-vector space isomorphism $\rho:V \to V'$ such that 
$$(\rho A - A' \rho)V=0, \qquad (\rho A^*(x) - {A^*}'(x') \rho)V=0.$$
	That is, the standard $T(x)$-module $V$ and the standard $T'(x')$-module $V'$ are $T$-isomorphic.
	\item[(iii)] There exist ordered bases $\mathcal{B}$ for $V$ and $\mathcal{B}'$ for $V'$ such that the matrix representing $A$ $($resp. $A^*(x)$$)$ with respect to $\mathcal{B}$ is equal to the matrix representing $A'$ $($resp. $A^{*'}(x')$$)$ with respect to $\mathcal{B}'$.
\end{itemize}
\end{lemma}
\begin{proof}
Routine. 
\end{proof}

\begin{lemma}\label{lem: G G' T-iso => int array}
If $\Gamma$ and $\Gamma'$ are $T$-isomorphic, then the intersection array of $\Gamma$ is equal to the intersection array of $\Gamma'$.
\end{lemma}
\begin{proof}
Fix vertices $x\in X$ and $x' \in X'$.
Let $V$ denote the standard $T(x)$-module of $\Gamma$ and $V'$ the standard $T'(x')$-module of $\Gamma'$.
Write $T=T(x)$, $T'=T'(x')$, $E^*_i=E^*_i(x)$, $E^{*'}_i=E^{*'}_i(x')$ for $0\leq i \leq D$.
Since $\Gamma$ and $\Gamma'$ are $T$-isomorphic, there exists a $\mathbb{C}$-vector space isomorphism $\rho:V \to V'$ such that $(\rho B - f(B) \rho)V=0$ for all $B\in T$, where $f$ is a $T$-algebra isomorphism from $T$ to $T'$.
Let $W$ denote the primary $T$-module.
Set $W'=\rho(W)$.
Then $W'$ is a $T'$-module since for all $B'\in T'$ we have $B'W'=B'\rho(W)=\rho(BW) \subseteq \rho(W)=W'$.
Clearly, the $T'$-module $W'$ is irreducible.
Moreover, $W'$ has endpoint $0$ since $E^*_0W\ne 0$ and $\rho(E^*_0W) = E^{*'}_0\rho(W)=E^{*'}_0W'\ne 0$.
By these comments, $W'$ is the primary $T'$-module; cf. \cite[Proposition 8.4]{2000Egge}.

Let $\bar \rho$ denote the restriction of $\rho$ to $W$.
Then $W$ and $W'$ are $T$-isomorphic because for all $w\in W$ and for all $B\in T$ we have $\bar\rho(Bw)=\rho(Bw)=f(B)\rho(w)=f(B)\bar\rho(w)$.
Recall the standard basis $\{E^*_i\mathds{1}\}^D_{i=0}$ for $W$. 
We observe that $\bar\rho(E^*_i\mathds{1}) = E^{*'}_i\bar\rho(\mathds{1})$ for $0 \leq i \leq D$ and the sequence $\{E^{*'}_i\bar\rho(\mathds{1})\}^D_{i=0}$ is a standard basis for $W'$.
Since $\bar\rho(AE^*_i\mathds{1}) = A'E^{*'}_i\bar\rho(\mathds{1})$ for $0 \leq i \leq D$, the matrix representing $A$ with respect to $\{E^*_i\mathds{1}\}^D_{i=0}$ is equal to the matrix representing $A'$ with respect to $\{E^{*'}_i\bar\rho(\mathds{1})\}^D_{i=0}$.
By \eqref{action: 3-term}, these matrices are tridiagonal whose entries are the intersection numbers of $\Gamma$ and $\Gamma'$, respectively. 
The result follows.
\end{proof}

\begin{lemma}\label{G =iso G' =>Spec(D1)=Spec(D2)}
Suppose that $\Gamma$ and $\Gamma'$ are $T$-isomorphic.
Then for $0 \leq i \leq D$ and for all vertices $x\in X$, $x' \in X'$, the spectrum of $\Delta_i(x)$ and the spectrum of $\Delta_i(x')$ are equal to each other.
\end{lemma}
\begin{proof}
Similar to Lemma \ref{pvt=>Spec(D1)=Spec(D2)}. Use Lemma \ref{lem:[A]_B=[A']_B'}.
\end{proof}

\section{Strongly regular graphs}\label{sec:SRG}
In this section, we summarize the results of \cite{2yama} and give a characterization of pseudo-vertex-transitive strongly regular graphs.
In addition, we give several examples of cospectral strongly regular graphs and discuss their pseudo-vertex transitivity.

\subsection{Preliminaries}
We begin by recalling the notion of a strongly regular graph.
Let $\Gamma$ be a $k$-regular graph with $n$ vertices.
We say $\Gamma$ is \emph{strongly regular} with parameters $(n,k,a,c)$ whenever each pair of adjacent vertices has the same number $a$ of common neighbors, and each pair of distinct non-adjacent vertices has the same number $c$ of common neighbors. 
Note that a connected strongly regular graph with parameters $(n,k,a,c)$ is distance-regular with diameter two and intersection array $\{k,k-a-1; 1,c\}$.

For the rest of this section, we denote by $\Gamma$ a connected strongly regular graph with parameters $(n,k,a,c)$.
The graph $\Gamma$ has exactly three eigenvalues $k$, $\sigma$, $\tau$ with
\begin{equation}\label{srg:eigval}
	\sigma = \frac{a-c+\sqrt{\pmb{D}}}{2}, \qquad
	\tau = \frac{a-c-\sqrt{\pmb{D}}}{2}, 
\end{equation}
where $\pmb{D}=(a-c)^2+4(k-c)$.
The multiplicities $m_\sigma$ and $m_\tau$ of $\sigma$ and $\tau$, respectively, are given by
\begin{equation}\label{srg:mult}
	m_\sigma = \frac{(n-1)\tau+k}{\tau-\sigma}, \qquad m_\tau = \frac{(n-1)\sigma+k}{\sigma-\tau}.
\end{equation}
Let $X$ denote the vertex set of $\Gamma$.
Fix a vertex $x \in X$.
Consider the distance partition $\{\Gamma_i(x)\}^2_{i=0}$ of $X$.
With respect to this partition, we write the adjacency matrix $A$ of $\Gamma$ in the partitioned matrix form:
\begin{equation}\label{adjmat}
	A = \begin{pmatrix}
	0 & \1^t & \mathbf{0} \\
	\1 & B_1 & N \\
	\mathbf{0} & N^t & B_2
	\end{pmatrix},
\end{equation}
where $B_i$ $(i=1,2)$ is the adjacency matrix of the $i$th subconstituent $\Delta_i(x)$ of $\Gamma$.
For notational convenience, we denote by
\begin{equation*}
	\widetilde A_1 := E^*_1AE^*_1 = \begin{pmatrix}
	0 & \mathbf{0} & \mathbf{0} \\
	\mathbf{0} & B_1 & \mathbf{0} \\
	\mathbf{0} & \mathbf{0} & \mathbf{0}
	\end{pmatrix},
	\qquad 
	\widetilde A_2 := E^*_2AE^*_2 = \begin{pmatrix}
	0 & \mathbf{0} & \mathbf{0} \\
	\mathbf{0} & \mathbf{0} & \mathbf{0} \\
	\mathbf{0} & \mathbf{0} & B_2
	\end{pmatrix},
\end{equation*}
where $E^*_i=E^*_i(x)$ $(i=1,2)$ is the $i$th dual primitive idempotent of $\Gamma$.
Note that $\Delta(x)$ is $a$-regular with $k$ vertices and that $\Delta_2(x)$ is  $(k-c)$-regular with  $n-k-1$ vertices.
It follows that $\Delta(x)$ and $\Delta_2(x)$ have trivial eigenvalues $a$ and $k-c$, respectively.
We say that an eigenvalue of $\Delta_i(x) (i=1,2)$ is \emph{local} if it is not equal to an eigenvalue of $\Gamma$ and has an eigenvector orthogonal to $\1$.
The following is a characterization of a local eigenvalue of $\Delta_i(x)$ $(i=1,2)$.

\begin{lemma}[cf.{\cite[Theorem 10.6.3]{god}}]\label{2sub}
Let $(i,j)=(1,2)$ or $(i,j)=(2,1)$.  
Then $\lambda$ is a local eigenvalue of $\Delta_i(x)$ if and only
if $a-c-\lambda$ is a local eigenvalue of  $\Delta_j(x)$, with equal multiplicities.
\end{lemma}

\begin{lemma}[cf. {\cite[Proposition 3.1, Lemma 3.2]{2yama}}]\label{int}
Let $\sigma, \tau$ be eigenvalues of $\Gamma$ as in \eqref{srg:eigval}.
Fix a vertex $x$ of $\Gamma$ and write $E^*_i=E^*_i(x)$ for $0 \leq i \leq 2$ and $T=T(x)$.
Let $A$ be the adjacency matrix of $\Gamma$ as in \eqref{adjmat}.
Let $V$ be the standard $T$-module.
Then the following (I) and (II) hold.
\begin{itemize}
	\item[(I)] Let $v\in E^*_1V$ be an eigenvector of $\widetilde A_1$ with an eigenvalue $\lambda$ such that $\<\1,v\>=0$.
	Then $E^*_2AE^*_1 v =0$ if and only if $\lambda \in \{\sigma, \tau\}$.
	Let $W$ be the subspace of $V$ spanned by $v, E^*_2A E^*_1v$.
	Then $W$ is a thin irreducible $T$-module.
	Moreover, the following (i) and (ii) hold.
	\item[(i)] Suppose $\lambda \notin\{\sigma, \tau\}$. 
	Then  $E^*_2AE^*_1 v$ is an eigenvector of $\widetilde A_2$ with an eigenvalue $\sigma+\tau-\lambda$.
	The $T$-module $W$ has endpoint $1$ and diameter $1$.
	The vector $v$ is a basis for $E^*_1W$ and $E^*_2AE^*_1 v$ is a basis for $E^*_2W$.
	The matrix representing $A$ with respect to a basis $\{v, E^*_2AE^*_1 v\}$ for $W$ is
	\begin{equation}\label{lem:mat(1)}
		\begin{pmatrix}
		\lambda & -(\lambda-\sigma)(\lambda-\tau)\\
		1 & \sigma+\tau-\lambda
		\end{pmatrix}.
	\end{equation}
	\item[(ii)] If $\lambda\in \{\sigma,\tau\}$, then $W=E^*_1W$. 
	The $T$-module $W$ has endpoint $1$ and diameter $0$. 
	The action of $A$ on $W$ is given by $Av=\lambda v$. 
	\item[(II)] Let $u\in E^*_2V$ be an eigenvector of $\widetilde A_2$ with an eigenvalue $\lambda'$ such that $\<\1,u\>=0$.
	Then $E^*_1AE^*_2 u =0$ if and only if $\lambda' \in \{\sigma, \tau\}$.
	Let $W$ be the subspace of $V$ spanned by $u, E^*_1A E^*_2u$.
	Then $W$ is a thin irreducible $T$-module.
	Moreover, the following (i) and (ii) hold.
	\item[(i)] Suppose $\lambda' \notin\{\sigma, \tau\}$. 
	Then  $E^*_1AE^*_2 u$ is an eigenvector of $\widetilde A_1$ with an eigenvalue $\sigma+\tau-\lambda'$.
	The $T$-module $W$ has endpoint $1$ and diameter $1$.
	The vector $u$ is a basis for $E^*_2W$ and $E^*_1AE^*_2 u$ is a basis for $E^*_1W$.
	The matrix representing $A$ with respect to a basis $\{E^*_1AE^*_2 u, u\}$ for $W$ is
	\begin{equation*}
		\begin{pmatrix}
		\lambda' & -(\lambda'-\sigma)(\lambda'-\tau)\\
		1 & \sigma+\tau-\lambda'
		\end{pmatrix}.
	\end{equation*}
	\item[(ii)] If $\lambda' \in \{\sigma,\tau\}$, then $W=E^*_2W$. 
	The $T$-module $W$ has endpoint $2$ and diameter $0$. 
	The action of $A$ on $W$ is given by $Au=\lambda' u$. 
\end{itemize}
\end{lemma}

Recall the first subconstituent space $E^*_1V$ of $\Gamma$. 
Note that $\dim(E^*_1V)=k$.
Let 
\begin{equation}\label{eq:eigvec til(A)}
	v_1, v_2, \ldots, v_r, v_{r+1}, \ldots, v_k
\end{equation}
denote a sequence of eigenvectors of $\widetilde A_1$ in $E^*_1V$ such that $v_1:=E^*_1\1$  and $\<v_i,v_j\>=0$ for $i\ne j$.
Observe that $\{v_i\}^k_{i=1}$ is an orthogonal basis for $E^*_1V$.
Let 
\begin{equation*}
	a=\lambda_1, \lambda_2, \ldots, \lambda_r, \lambda_{r+1}, \ldots, \lambda_k
\end{equation*}
denote the eigenvalues of $\tilde{A}_1$ corresponding to the eigenvectors in \eqref{eq:eigvec til(A)}, respectively.
Assume that $\lambda_i \notin \{\sigma, \tau\}$ for $1\leq i \leq r$ and $\lambda_j \in \{\sigma, \tau\}$ for $r+1\leq j \leq k$.  
We consider the second subconstituent space $E^*_2V$ of $\Gamma$.
Note that $\dim(E^*_2V)=n-k-1$.
For each vector $v_i$ $(1\leq i \leq r)$ in \eqref{eq:eigvec til(A)}, abbreviate $u_i=E^*_2AE^*_1v_i$.
Observe that $u_1$ is an eigenvector of $\widetilde{A}_2$ with eigenvalue $k-c$.
By Lemma \ref{int}(I)(i), $u_i$ $(2\leq i \leq r)$ is an eigenvector of $\widetilde A_2$ associated with the eigenvalue $\sigma+\tau-\lambda_i \notin \{\sigma, \tau\}$.
We choose the rest of the eigenvectors of $\widetilde A_2$ in $E^*_2V$, denoted by
$$
	u_{r+1}, u_{r+2}, \ldots, u_{s}, 
$$
where $s:= n-k-1$, such that $\{u_i\}^s_{i=1}$ spans $E^*_2V$ and $\<u_i, u_j\>=0$ for $1\leq i\ne j \leq s$.
Observe that $\{u_i\}^s_{i=1}$ is a basis for $E^*_2V$. 
Let $\lambda'_h$ denote the eigenvalue of $\widetilde A_2$ associated with $u_h$ for $r+1\leq h \leq s$.
 By Lemma \ref{int}(II), we find $\lambda'_h \in \{\sigma, \tau\}$.
Let $W_i=Tv_i$ for $1\leq i \leq k$ and let $W'_h=Tu_h$ for $r+1 \leq h \leq s$.
We say $W_i$ (resp. $W'_h$) is \emph{associated with} $\lambda_i$ (resp. $\lambda'_h$).
We find that $W_i$ and $W'_h$ are thin irreducible $T$-modules.
In the following table, we display all thin irreducible $T$-modules of $\Gamma$:
\begin{equation}\label{eq:thinT-modSRG}
\renewcommand{\arraystretch}{1.2}
\begin{tabular}{lcccc}
	\hline
	thin irreducible $T$-module & basis & dimension & endpoint & diameter  \\
	\hline
	\qquad $W_1$ & $\{E^*_nAE^*_1v_1\}^2_{n=0}$  & 3 & 0 & 2 \\
	\qquad $W_i$ $(2 \leq i \leq r)$ & $v_i$, $E^*_2AE^*_1v_i$ & 2 & 1 & 1 \\
	\qquad $W_j$ $(r+1 \leq j \leq k)$ & $v_j$ & 1 & 1 & 0 \\
	\qquad $W'_h$ $(r+1 \leq h \leq s)$& $u_h$ &1  & 2 & 0\\
	\hline
\end{tabular}
\end{equation}
Since $T$ is semisimple, the standard $T$-module $V$ decomposes into an orthogonal direct sum of thin irreducible $T$-submodules: 
\begin{equation}\label{eq:SRG V}
	V = W_1 \oplus (\bigoplus^r_{i=2} W_i) \oplus (\bigoplus^k_{j=r+1} W_j) \oplus (\bigoplus^s_{h=r+1} W'_h).
\end{equation}

The following proposition classifies thin irreducible $T$-modules of $\Gamma$ up to isomorphism.
\begin{proposition}[cf.{\cite[Lemma 3.4]{2yama}}]\label{prop:classT-mod}
Let $W_i$ (resp. $W'_h$) denote an irreducible $T$-module associated with an eigenvalue $\lambda_i$ (resp. $\lambda'_h$) as in \eqref{eq:thinT-modSRG}.
Then the following (i) and (ii) hold.
\begin{itemize}
	\item[(i)] For  $2\leq i,j \leq k$, $W_i$ and $W_j$ are isomorphic as $T$-modules if and only if $\lambda_i=\lambda_j$.
	\item[(ii)] For $r+1\leq h,l \leq s$, $W'_h$ and $W'_l$ are isomorphic as $T$-modules if and only if $\lambda'_h=\lambda'_l$.
\end{itemize}
\end{proposition}

Let $\ell_1$ (resp. $\ell_2$) denote the number of distinct eigenvalues of $\Delta(x)$ (resp. $\Delta_2(x)$) contained in $\{\sigma, \tau\}$ with respect to the eigenvectors orthogonal to $E^*_1\1$ (resp. $E^*_2\1$).
Let $\ell_1'$ (resp. $\ell_2'$) denote the number of distinct eigenvalues of $\Delta(x)$ (resp. $\Delta_2(x)$) not contained in $\{\sigma, \tau\}$ with respect to the eigenvectors orthogonal to $E^*_1\1$ (resp. $E^*_2\1$).
By Proposition \ref{prop:classT-mod}, the isomorphism classes of irreducible $T$-modules of $\Gamma$ are determined by $\Spec(\Gamma)$, $\Spec(\Delta(x))$, and $\Spec(\Delta_2(x))$.
It follows that there are $\ell_1+\ell_2$ irreducible $T(x)$-modules of dimension one up to isomorphism, and there are $\ell_1'$ $(=\ell_2')$ thin irreducible $T(x)$-modules of dimension two up to isomorphism. 
There is only one thin irreducible $T(x)$-module of dimension three, namely the primary $T(x)$-module.
By Proposition \ref{Prop:Wedderburn} and \eqref{eq:SRG V}, the algebra $T(x)$ is isomorphic to the semisimple algebra
\begin{equation}\label{eq:ds of mat algebra}
	\big(\mathrm{Mat}_{1}(\mathbb{C})^{\oplus \ell_1}\big) \oplus
	\big(\mathrm{Mat}_{1}(\mathbb{C})^{\oplus \ell_2}\big) \oplus
	\big(\mathrm{Mat}_{2}(\mathbb{C})^{\oplus \ell_1'}\big) \oplus
	\mathrm{Mat}_{3}(\mathbb{C}).
\end{equation}
We call $(\ell_1,  \ell_1', \ell_2, \ell_2')$ the \emph{dimension sequence} of $T(x)$.
Note that the dimension sequence of $T(x)$ depends on the base vertex $x$ of $\Gamma$.

\begin{lemma}[cf. {\cite[Theorem 1.1]{2yama}}]\label{dim}
Let $x$ be a base vertex of $\Gamma$ and let $T(x)$ be the Terwilliger algebra of $\Gamma$ with respect to $x$. 
Then
\begin{align*}
	\dim T(x) 
		& =\ell_1+\ell_2+4 \ell_1'+9 \\
		&=\ell_1+\ell_2+4 \ell_2'+9,
\end{align*}
where $(\ell_1,  \ell_1', \ell_2, \ell_2')$ is the dimension sequence of $T(x)$.
\end{lemma}
\begin{proof}
By \eqref{eq:ds of mat algebra} and Proposition \ref{Prop:Wedderburn}.
\end{proof}

\subsection{A characterization of pseudo-vertex-transitive strongly regular graphs}

In this subsection, we characterize pseudo-vertex-transitive strongly regular graphs. 
First, we discuss a relationship between $\mathrm{Spec}(\Delta(x))$ and $\mathrm{Spec}(\Delta_2(x))$ of $\Gamma$.

\begin{lemma}\label{specdet}
Let $k, \sigma, \tau$ be eigenvalues of $\Gamma$ as in \eqref{srg:eigval} with multiplicities $1, m_\sigma, m_\tau$, respectively.  
Fix a vertex $x$ of $\Gamma$.
The spectrum of the second subconstituent $\Delta_2(x)$ of $\Gamma$ is determined by the spectrum of the local graph $\Delta(x)$ of $\Gamma$ and the parameters $(n,k,a,c)$.
\end{lemma}
\begin{proof}
Let us denote the spectrum of $\Delta(x)$ by 
$$
\{a^1, \sigma^{f_\sigma}, \tau^{f_\tau}, \lambda_1^{m_1}, \ldots,\lambda_s^{m_s}\},
$$ 
for some $s$. Here, it is possible that $f_\sigma=0$ and/or $f_\tau=0$. 
By Lemma \ref{2sub}, we can write the spectrum of $\Delta_2(x)$ as
$$
\{(k-c)^1, \sigma^{g_\sigma}, \tau^{g_\tau}, (a-c-\lambda_1)^{m_1}, \ldots,(a-c-\lambda_s)^{m_s}\},
$$
where it is possible that $g_\sigma=0$ and/or $g_\tau=0$. 
Since $|\Gamma_1(x)|=k$ and $|\Gamma_2(x)|=n-k-1$, we have
\begin{equation}\label{eq:sum mult}
	1+f_\sigma+f_\tau+m_1+\cdots+m_s=k, \qquad
	1+g_\sigma+g_\tau+m_1+\cdots + m_s=n-k-1.
\end{equation}
Using the equations in \eqref{eq:sum mult}, we obtain
\begin{equation}\label{pf:lem eq(1)}
	g_\sigma+g_\tau = n-2k-1+f_\sigma+f_\tau.
\end{equation}
Since the sum of all the eigenvalues of $\Delta_2(x)$ is the trace of $\widetilde A_2$, which is zero, it follows
\begin{equation}\label{pf:lem eq(2)}
	\sigma g_\sigma+\tau g_\tau=-(k-c)-m_1(a-c-\lambda_1)-\ldots-m_s(a-c-\lambda_s).
\end{equation}
Solve the system of equations \eqref{pf:lem eq(1)} and \eqref{pf:lem eq(2)} for $g_\sigma$ and $g_\tau$ to obtain
\begin{equation}\label{pf:mult 2nd subc}
	g_\sigma = -k + m_\sigma + f_\tau, \qquad g_\tau=-k + m_\tau +f_\sigma.
\end{equation}
The result follows.
\end{proof}
\begin{remark}
By \eqref{pf:mult 2nd subc} we have
$$
	f_\sigma = k - m_\tau + g_\tau  , \qquad f_\tau=k-m_\sigma+g_\sigma.
$$	
Thus, the spectrum of $\Delta(x)$ is also determined by the spectrum of $\Delta_2(x)$.
\end{remark}

\begin{example}
Consider the Johnson graph $J(8,2)$ with vertex set $X={\Omega\choose2}$, the collection of all $2$-subsets of $\Omega=\{1,2,\ldots, 8\}$.
The graph $J(8,2)$ is strongly regular with parameters $(28, 12, 6, 4)$.
Fix a vertex $x \in X$. 
The spectrum of $J(8,2)$ is $\{12^1, 4^7, -2^{20}\}$ and the spectrum of the local graph $\Delta(x)$ is given by $\{6^1, 4^1, -2^5, 0^5\}$.
Using Lemma \ref{specdet}, we obtain the spectrum of $\Delta_2(x)$ of $J(8,2)$:
$$
	\{8^1, 4^0, -2^9, 2^5\}.
$$
\end{example}

\begin{proposition}\label{prop:SRGthin}
A connected strongly regular graph $\Gamma$ is thin.
\end{proposition}
\begin{proof}
By \eqref{eq:thinT-modSRG}, $\Gamma$ is thin with respect to $x$.
By Proposition \ref{prop:classT-mod} and Lemma \ref{specdet}, the result follows.
\end{proof}

Recall $E^*_i=E^*_i(x)$ the $i$th dual primitive idempotent of $\Gamma$ ($i=0,1,2$).
Consider another connected strongly regular graph $\Gamma'$ with parameters $(n',k',a',c')$ and vertex set $X'$. 
Fix $x' \in X'$.
Let $A'$ be the adjacency matrix of $\Gamma'$ and $A^{*'}=A^{*'}(x')$ the dual adjacency matrix of $\Gamma'$ with respect to $x'$.
Write $E^{*'}_i = E^{*'}_i(x')$ for $0\leq i \leq 2$ and $T'=T'(x')$.
With these notations, we state the following two lemmas.

\begin{lemma}\label{lem:Tv=T'v' iso primary}
Suppose that the parameters of $\Gamma$ and $\Gamma'$ are same, that is, $(n,k,a,c)=(n',k',a',c')$.
Set $v= E^*_1\1$ and $v'=E^{*'}_1\1$.
Then $Tv$ and $T'v'$ are $T$-isomorphic.
\end{lemma}
\begin{proof}
Observe that $Tv$ is the primary $T$-module with a basis $\{E^*_iAE^*_1v\}^2_{i=0}$; see \eqref{eq:thinT-modSRG}.
Define the vectors $\{v_i\}^2_{i=0}$ by
$$
	v_0 = E^*_0AE^*_1v, \qquad 
	v_1 = a^{-1}E^*_1AE^*_1v, \qquad
	v_2 = c^{-1}E^*_2AE^*_1v.
$$
Observe that $v_i=E^*_i\1$ ($0\leq i \leq 2$), and thus $\{v_i\}^2_{i=0}$ is a standard basis for $Tv$.
Similarly, we define the vectors $\{v'_i\}^2_{i=0}$ by 
$$
	v'_0 = E^{*'}_0A'E^{*'}_1v', \qquad 
	v'_1 = a^{-1}E^{*'}_1A'E^{*'}_1v', \qquad
	v'_2 = c^{-1}E^{*'}_2A'E^{*'}_1v'.
$$
Then $\{v'_i\}^2_{i=0}$ is a standard basis for the primary $T'$-module $T'v'$.
Define a vector space isomorphism $\rho: Tv \to T'v'$ that sends $v_i$ to $v'_i$.
Since $\Gamma$ and $\Gamma'$ have same parameters $(n,k,a,c)$, from \eqref{action: 3-term} the matrix representing $A$ with respect to $\{v_i\}^2_{i=0}$ and the matrix representing $A'$ with respect to $\{v'_i\}^2_{i=0}$ are the same as the tridiagonal matrix with entries of the intersection numbers of $\Gamma$:
$$
	\begin{pmatrix}
	0 & k & 0 \\
	1 & a & k-a-1 \\
	0 & c & k-c
	\end{pmatrix}.
$$
Therefore, we have $(\rho A - A' \rho)v_i=0$ for $0\leq i \leq 2$.
Moreover, we readily check that $(\rho E^*_j - E^{*'}_j \rho)v_i=0$ for $0\leq i,j \leq 2$.
The result follows.
\end{proof}

\begin{lemma}\label{lem:Tv=T'v' iso}
Suppose that the parameters of $\Gamma$ and $\Gamma'$ are same, that is, $(n,k,a,c)=(n',k',a',c')$.
Let $i \in \{1,2\}$.
Let $v\in E^*_iV$ denote an eigenvector of $E^*_iAE^*_i$ with an eigenvalue $\lambda$ such that $\<\1,v\>=0$.
Let $v' \in E^{*'}_iV$ denote an eigenvector of $E^{*'}_iA'E^{*'}_i$ with an eigenvalue $\lambda'$ such that $\<\1,v'\>=0$. 
Then $\lambda=\lambda'$ if and only if $Tv$ and $T'v'$ are $T$-isomorphic.
\end{lemma}

\begin{proof}
Set $i=1$. Suppose that $\lambda=\lambda'$.
If $\lambda \notin \{\sigma, \tau\}$, then by Lemma \ref{int}(I)(i) $Tv$ has a basis $v, E^*_2AE^*_1v$ and $T'v'$ has a basis $v', E^{*'}_2A'E^{*'}_1v'$.
Abbreviate $w=E^*_2AE^*_1v$ and $w'=E^{*'}_2A'E^{*'}_1v'$.
Define a vector space isomorphism $\rho: Tv \to T'v'$ by $\rho(v)=v'$ and $\rho(w)=w'$.
By Lemma \ref{int}(I)(i) and since $\lambda=\lambda'$, the matrix representing $A$ with respect to a basis $v, w$ and the matrix representing $A'$ with respect to a basis $v', w'$ are both equal to the matrix \eqref{lem:mat(1)}.
From this, it follows that $(\rho A - A'\rho)Tv=0$.
It is clear that $(\rho E^*_i - E^{*'}_i \rho)Tv=0$ for $0\leq i \leq 2$.
Thus, $Tv$ and $T'v'$ are $T$-isomorphic.
If $\lambda \in \{\sigma, \tau\}$, use Lemma \ref{int}(I)(ii) in a similar manner as demonstrated above to get the result.
Conversely, if $Tv$ and $T'v'$ are $T$-isomorphic, then the matrix representing $A$ with respect to a basis $v, w$ must be the same as the matrix representing $A'$ with respect to a basis $v', w'$. By this and using Lemma \ref{int}(I), we have $\lambda=\lambda'$.
For $i=2$, use Lemma \ref{int}(II). 
The result follows.
\end{proof}

We now characterize pseudo-vertex-transitive strongly regular graphs.

\begin{theorem}\label{prop:char_p.v.t}
A connected strongly regular graph $\Gamma$ is pseudo-vertex-transitive if and only if for every pair of vertices $x,y \in X$, the spectrum of the local graph $\Delta(x)$ of  $\Gamma$ at $x$ is equal to the spectrum of the local graph $\Delta(y)$ of $\Gamma$ at $y$.
\end{theorem}

\begin{proof}
If $\Gamma$ is pseudo-vertex-transitive, by Lemma \ref{pvt=>Spec(D1)=Spec(D2)} it immediately follows that $\Spec(\Delta(x))=\Spec(\Delta(y))$.
Conversely, for $x,y \in X$ suppose that $\Spec(\Delta(x))=\Spec(\Delta(y))$.
By Lemma \ref{specdet}, it follows that $\Spec(\Delta_2(x))=\Spec(\Delta_2(y))$.
Using these, we will show that the standard $T(x)$-module and the standard $T(y)$-module are isomorphic.
First, choose an eigenvalue $\lambda \in \Delta(x)=\Delta(y)$.
Let $v \in E^*_1(x)V$ denote an eigenvector of $E^*_1(x)AE^*_1(x)$ corresponding to $\lambda$, and let $v'\in E^*_1(y)V$ denote an eigenvector of $E^*_1(y)AE^*_1(y)$ corresponding to $\lambda$.
If $\lambda=a$, then $Tv$ and $T'v'$ are the primary modules of $T(x)$ and $T(y)$, respectively.
By Lemma \ref{lem:Tv=T'v' iso primary}, $T(x)v$ and $T(y)v'$ are isomorphic.
If $\lambda\ne a$, we observe that $T$-modules $T(x)v$ and $T(y)v'$ have both endpoint $1$.
By Lemma \ref{lem:Tv=T'v' iso}, $T(x)v$ and $T(y)v'$ are isomorphic.
Next, choose an eigenvalue $\mu \in \Delta_2(x)=\Delta_2(y)$ such that $\mu \in \{\sigma, \tau\}$.
Let $u \in E^*_2(x)V$ denote an eigenvector of $E^*_2(x)AE^*_2(x)$ corresponding to $\mu$, and let $u'\in E^*_2(y)V$ denote an eigenvector of $E^*_2(y)AE^*_2(y)$ corresponding to $\mu$.
Observe that both $T$-modules $T(x)v$ and $T(y)v'$ have endpoint $2$ and diameter $0$.
By Lemma \ref{lem:Tv=T'v' iso}, $T(x)u$ and $T(y)u'$ are isomorphic.
By these comments and \eqref{eq:SRG V}, the standard modules for $T(x)$ and $T(y)$ are isomorphic.
Since $x,y$ is an arbitrary pair of vertices in $X$, $\Gamma$ is pseudo-vertex-transitive.
\end{proof}

In the following, we characterize connected strongly regular graphs which are $T$-isomorphic.

\begin{theorem}\label{T-isomorphism}
Let $\Gamma$ and $\Gamma'$ be connected strongly regular graphs with parameters $(n,k,a,c)$ and $(n',k',a',c')$, respectively.
Then $\Gamma$ and $\Gamma'$ are $T$-isomorphic if and only if both (i) $(n,k,a,c)=(n',k',a',c')$ and (ii) for every pair of vertices $x\in X$ and $x' \in X'$ the spectrum of the local graph $\Delta(x)$ of  $\Gamma$ is equal to the spectrum of the local graph $\Delta'(x')$ of $\Gamma'$.
\end{theorem}

\begin{proof}
Suppose that $\Gamma$ and $\Gamma'$ are $T$-isomorphic.
By Lemma \ref{lem: G G' T-iso => int array}, the intersection arrays of $\Gamma$ and $\Gamma'$ are the same. 
From this, it follows $(n,k,a,c)=(n',k',a',c')$.
Also, by Lemma \ref{G =iso G' =>Spec(D1)=Spec(D2)}, it follows $\Spec(\Delta(x))=\Spec(\Delta'(x'))$.
Conversely, suppose that both (i) and (ii) hold.
We claim that $\Spec(\Delta_2(x))=\Spec(\Delta_2'(x'))$.
We denote by 
\begin{align*}
	\Spec{(\Delta_2(x))} & =\{(n-k)^1, \sigma^{\alpha_\sigma}, \tau^{\alpha_\tau}, \lambda_1^{m_1}, \ldots, \lambda_s^{m_s}\}, \\
	\Spec{(\Delta'_2(x'))} & =\{(n-k)^1, \sigma^{\beta_\sigma}, \tau^{\beta_\tau}, {\lambda_1'}^{n_1}, \ldots, {\lambda'_t}^{n_t}\},
\end{align*}
where $\sigma, \tau$ are eigenvalues of $\Gamma$ from \eqref{srg:eigval}.
Since $\Spec(\Delta(x))=\Spec(\Delta'(y))$ and by Lemma \ref{2sub}, it follows that $s=t$ and $m_i=n_i$ and $\lambda_i=\lambda'_i$ for $1\leq i \leq s$.
The sum of multiplicities of all eigenvalues of $\Delta_2(x)$ is $n-1-k$, and the sum of all eigenvalues of $\Delta_2(x)$ is zero.
Moreover, the sum of multiplicities of all eigenvalues of $\Delta'_2(y)$ is $n'-1-k'(=n-1-k)$ and the sum of all eigenvalues of $\Delta'_2(y)$ is zero.
From these comments, we have
\begin{align*}
	\alpha_\sigma + \alpha_\tau &= \beta_\sigma+\beta_\tau, \\
	\sigma \alpha_\sigma + \tau \alpha_\tau &= \sigma \beta_\sigma+\tau\beta_\tau.
\end{align*}
Solving these equations using the fact $\sigma\ne \tau$, we find that $\alpha_\sigma=\beta_\sigma$ and $\alpha_\tau=\beta_\tau$.
The desired claim follows.
Now, we show that the standard $T(x)$-module and the standard $T'(x')$-module are $T$-isomorphic. It is similar to the proof of Theorem \ref{prop:char_p.v.t}; use Lemma \ref{lem:Tv=T'v' iso primary} and Lemma \ref{lem:Tv=T'v' iso}.
Since $x\in X$ and $x'\in X'$ are arbitrary, it follows that $\Gamma$ and $\Gamma'$ are $T$-isomorphic.
\end{proof}

\subsection{Examples}

In this subsection, we give some concrete examples of cospectral strongly regular graphs and discuss their pseudo-vertex transitivity.

\begin{example}
Let $\Gamma$ denote the \emph{Shrikhande graph}, a Cayley graph of $\mathbb{Z}_4\times\mathbb{Z}_4$ relative to the generating set $\{\pm(1,0), \pm(0,1), \pm(1,1)\}$.
Let $\Gamma'$ denote the \emph{$4\times 4$-grid graph}, the line graph of $K_{4,4}$.
The graphs $\Gamma$ and $\Gamma'$ are cospectral strongly regular with parameters $(16,6,2,2)$ and the eigenvalues $k=6, \sigma=2, \tau=-2$.
Moreover, $\Gamma$ and $\Gamma'$ are both vertex-transitive and thus pseudo-vertex-transitive.
The following table shows the spectrum of the local graph $\Delta_G$ of $G\in \{\Gamma, \Gamma'\}$ and the dimension of the Terwilliger algebra $T_G$ of $G$.
\begin{gather*}
\text{Strongly regular graph with $(16,6,2,2)$}\\
\renewcommand{\arraystretch}{1.2}
\begin{tabular}{cccccc}
\hline
graph $G$ & \quad & $\Spec(\Delta_G)$ &\quad & $\dim{T_G}$\\
\hline 
$\Gamma$ & \quad & $\{2^1, 1^2, (-1)^2, (-2)^1\}$ & \quad & $20$\\
$\Gamma'$ & \quad & $\{2^2,(-1)^4\}$ & \quad& $15$ \\
\hline
\end{tabular}.
\end{gather*}
By Theorem \ref{T-isomorphism}, it follows that $\Gamma$ and $\Gamma'$ are not $T$-isomorphic.
\end{example}

\begin{example}[cf. {\cite[Example 4.2]{2yama}}]\label{ex:chang}
Recall the Johnson graph $\Gamma=J(8,2)$ with vertex set $X={\Omega \choose 2}$, where $\Omega=\{1,2,\ldots, 8\}$.
The automorphism group $\mathrm{Aut}(\Gamma)$\footnote{Note that $\mathrm{Aut}(\Gamma)=S_8$.} acts on $X$ transitively, and hence $\Gamma$ is pseudo-vertex-transitive.
We recall the \emph{Chang graphs}. 
The three Chang graphs $\Gamma'$, $\Gamma''$, $\Gamma'''$ can be obtained from $\Gamma$ by Seidel switching with respect to one of the sets (cf. \cite[pp.105]{1989BCN})
\begin{itemize}[itemsep=0.5pt]
	\item[(i)] $\{\{1,5\},\{2,6\},\{3,7\},\{4,8\}\}$ for $\Gamma'$, 
	\item [(ii)] $\{\{1,2\}, \{2,3\},\{3,4\},\{4,5\},\{5,6\},\{6,7\},\{7,8\},\{8,1\}\}$ for $\Gamma''$,
	\item [(iii)] $\{\{1,2\}, \{2,3\},\{3,1\},\{4,5\},\{5,6\},\{6,7\},\{7,8\},\{8,4\}\}$ for $\Gamma'''$.
\end{itemize}
Note that none of three Chang graphs (i)--(iii) is vertex-transitive.
For case (i), the action of $\mathrm{Aut}(\Gamma')$ on $X$ has two orbits:
\begin{equation*}
	U'_1=\{\{1,5\}, \{2,6\},\{3,7\},\{4,8\}\}, \qquad U'_2=X \setminus U_1'.
\end{equation*}
For case (ii), the action of $\mathrm{Aut}(\Gamma'')$ on $X$ has two orbits:
\begin{equation*}
	U''_1=\{\{1,5\}, \{2,6\},\{3,7\},\{4,8\}\}, \qquad U''_2=X \setminus U_1''.
\end{equation*}
For case (iii), the action of $\mathrm{Aut}(\Gamma''')$ on $X$ has three orbits\footnote{In \cite[Example 4.2]{2yama}, the orbit $U'''_1$ was missed.}
\begin{equation*}
	U'''_1=\{\{1,2\}, \{2,3\},\{3,1\}\}, \quad 
	U'''_2=\{(i,j) \mid 4 \leq i<j\leq 8\}, \quad
	U'''_3=X\setminus \{U'''_1 \cup U'''_2\}.
\end{equation*}


We remark that the Johnson graph $J(8,2)$ and three Chang graphs (i)--(iii) are cospectral strongly regular with parameters $(28,12,6,4)$ and the eigenvalues $k=12$, $\sigma=4$, $\tau=-2$.
In the following table we display the spectrum of the local graph $\Delta_G$ of $G \in \{\Gamma, \Gamma', \Gamma'', \Gamma'''\}$ and the dimension of the Terwilliger algebra $T_G$ of $G$.
\begin{gather*}
\text{Strongly regular graph with $(28,12,6,4)$}\\
\renewcommand{\arraystretch}{1.3}
\begin{tabular}{cccccc}
\hline
$G$ & base vertex $x$ & $\Spec(\Delta_G(x))$ & $\dim{T_G(x)}$ \\
\hline
$\Gamma$ & $x \in X$ & $\{6^1, 0^5, 4^1, (-2)^5\}$ & $16$\\
\hline
\multirow{2}{*}{$\Gamma'$}	
	& $x \in U'_1$ &  $\{6^1, 2^3, 0^2, (-2)^6\}$ & $20$  \\
	& $x \in U'_2$ & $\{6^1, (1+\sqrt{5})^1, 2^1, 0^3, (1-\sqrt{5})^1,(-2)^5\}$ & $27$\\
\hline
\multirow{2}{*}{$\Gamma''$}	
	& $x \in U''_1$ &  $\{6^1, (1+\sqrt{3})^2, 0^2, (1-\sqrt{3})^2,(-2)^5\}$ & $23$ \\
	& $x \in U''_2$ & $\{6^1, (1+\sqrt{3})^1, 2^1, \sqrt{2}^1, 0^1, (1-\sqrt{3)}^1, (-\sqrt{2})^1, (-2)^5\}$ & 35  \\
\hline
\multirow{3}{*}{$\Gamma'''$}	
	& $x \in U'''_1$ & $\{6^1, 3^1, (\frac{1+\sqrt{5}}{2})^2,  (\frac{1-\sqrt{5}}{2})^2, (-1)^1, (-2)^5\}$ & $27$  \\
	& $x \in U'''_2$ &  $\{6^1, (\frac{1+\sqrt{13}}{2})^2, 1^2,  (\frac{1-\sqrt{13}}{2})^2, (-2)^5\}$ & $23$\\
	& $x \in U'''_3$ & $\{6^1, (1+\sqrt{3})^1, 2^1, \sqrt{2}^1, 0^1, (1-\sqrt{3})^1, (-\sqrt{2})^1, (-2)^5\}$ & $35$ \\
\hline
\end{tabular}
\end{gather*}
From this table and by Proposition \ref{prop:char_p.v.t}, we find that the three Chang graphs are not pseudo-vertex-transitive.
We finish this example with comments. 
\begin{itemize}[itemsep=0.5pt]
	\item[(i)] For $x \in U'_2$ and $y\in U'''_1$, the Terwilliger algebras $T_{\Gamma'}(x)$ and  $T_{\Gamma'''}(y)$ have the same dimension $27$, but they are not $T$-algebra isomorphic since their local spectra are not equal to each other. 
	Note that the dimension sequences of $T_{\Gamma'}(x)$ and  $T_{\Gamma'''}(y)$ are equal to $(1,4,1,4)$, and thus they are semisimple algebra isomorphic to each other.
	\item[(ii)] In a similar way to (i), for $x \in U''_1$ and $y\in U'''_2$, the Terwilliger algebras $T_{\Gamma''}(x)$ and $T_{\Gamma'''}(y)$ are not $T$-algebra isomorphic, but their dimension sequences are equal to $(1,3,1,3)$. Thus they are semisimple algebra isomorphic to each other.
	\item[(iii)] For $x \in U''_2$ and $y\in U'''_3$, the local graphs $\Delta_{\Gamma''}(x)$ and $\Delta_{\Gamma'''}(y)$ have the same spectrum. 
	Therefore, the Terwilliger algebras $T_{\Gamma''}(x)$ and  $T_{\Gamma'''}(y)$ are $T$-algebra isomorphic.
\end{itemize}
\end{example}

\begin{example}\label{ex:GQ}
A \emph{generalized quadrangle} is an incidence structure such that:
(i) any two points are on at most one line, and hence any two lines meet in at most one point,
(ii) If $p$ is a point not on a line $L$, then there is a unique point $p'$ on $L$ such that $p$ and $p'$ are collinear.
If every line contains $s+1$ points, and every point lies on $t+1$ lines, we say that the generalized quadrangle has order $(s,t)$, denoted by $GQ(s,t)$.

The \emph{point graph} of a generalized quadrangle is the graph with the points of the quadrangle as its vertices, with two points adjacent if and only if they are collinear. 
The point graph $\Gamma$ of a $GQ(s,t)$ is strongly regular with parameters
$((s+1)(st+1), s(t+1), s-1, t+1)$ and eigenvalues $k=s(t+1), \sigma=s-1, \tau=-t-1$, with respective multiplicities
\begin{equation*}
	1, \quad \frac{st(s+1)(t+1)}{s+t}, \quad \frac{s^2(st+1)}{s+t}.
\end{equation*}
By construction, $\Gamma$ is locally a disjoint union of the $(t+1)$ cliques of size $s$.
Thus the spectrum of a local graph $\Delta$ of $\Gamma$ is determined by $s, t$:
\begin{equation*}
	\Spec(\Delta)=\{(s-1)^{t+1}, {-1}^{(s-1)(t+1)}\}.
\end{equation*}
From these comments and by Theorem \ref{prop:char_p.v.t}, it follows that $\Gamma$ is pseudo-vertex-transitive.
Moreover, the point graph $\Gamma'$ of a $GQ(s',t')$ is $T$-isomorphic to $\Gamma$ if and only if $(s,t)=(s',t')$.
The dimension of the Terwilliger algebra $T(x)$ of $\Gamma$ with respect to a vertex $x$ is given as follows.
\begin{equation*}
	\dim T(x)=
	\begin{cases}
	10, & \enskip \text{if} \enskip t=s=1;\\
	15, &\enskip \text{if}  \enskip t=1,~s \neq 1, \enskip \text{or} \enskip t \neq 1,~s\neq 1,~s^2=t;\\
	11, &\enskip \text{if} \enskip  t \neq 1,~s=1;\\
	16, &\enskip \text{if} \enskip t \neq 1,~s\neq 1,~s^2\neq t.
	\end{cases}
\end{equation*}
\end{example}
\noindent

We finish this example with a comment.
The isomorphism class of the Terwilliger algebra for the point graph of a $GQ(s,t)$ only depends on $s$ and $t$, and thus it is pseudo-vertex-transitive.
However, there exist generalized quadrangles whose point graphs are not vertex-transitive. For example, there are at least three pairwise non-isomorphic $GQ(q,q^2)$ of which at least one is not vertex-transitive, provided that $q=2^{2h+1}$, $h\ge 1$; see \cite[Page 45]{pt}.

\section{Preliminaries: antipodal distance-regular graphs}\label{sec:antipodal 2-cover}
In this section, we recall some preliminaries concerning antipodal distance-regular graphs, which will be used in the next sections.
Let $\Gamma$ denote a distance-regular graph with vertex set $X$ and diameter $D\geq 3$. 
The graph $\Gamma$ is called \emph{antipodal} whenever the vertices at distance $D$ from a given vertex are all at distance $D$ from each other.
Note that $\Gamma$ is antipodal if and only if $\Gamma$ has the intersection array $\{b_0, b_1, \ldots, b_{D-1}; c_1, c_2, \ldots, c_D\}$, where $b_i=c_{D-i}$ for $0 \leq i \leq D-1$, except possibly $i=\lfloor D/2\rfloor$; cf. \cite[Proposition 4.2.2]{1989BCN}.
Note also that being at distance $D$ or zero induces an equivalence relation on $X$.
With respect to this equivalence relation, the equivalence classes are called \emph{antipodal classes}. 
We say $\Gamma$ is an \emph{antipodal $r$-cover} if the equivalence class has size $r$.
Let $\Gamma$ be an antipodal $2$-cover (or \emph{double cover}).
Then for each vertex $x \in X$ there is a unique vertex $y \in X$ with $\partial(x,y)=D$.
Such a unique vertex is called the \emph{antipode} of $x$, denoted by $\hat{x}$.

\begin{lemma}\label{T(x)-mod T(x^)-mod}
Let $\Gamma$ be an antipodal double cover with diameter $D\geq 3$.
Let $x$ be a vertex of $\Gamma$ and $\hat{x}$ the antipode of $x$.
Then $W$ is an irreducible $T(x)$-module if and only if $W$ is an irreducible $T(\hat{x})$-module.
\end{lemma}
\begin{proof}
Let $W$ denote an irreducible $T(x)$-module.
Then $W$ is invariant under the actions of $A$ and $E^*_j(x)$ for $0 \leq j \leq D$.
Since $\Gamma$ is an antipodal double cover, it follows that $E^*_j(x) = E^*_{D-j}(\hat{x})$ for $0 \leq j\leq D$.
Thus $W$ is invariant under the actions of $A$ and $E^*_h(\hat{x})$ for $0 \leq h \leq D$.
The result follows.
\end{proof}

\noindent
Suppose that $\Gamma$ is $Q$-polynomial with respect to the ordering $E_0, E_1, \ldots, E_D$. 
Then $\Gamma$ is antipodal if and only if $\Gamma$ is \emph{dual bipartite}, that is, the Krein parameters of $\Gamma$ satisfy $q^i_{1i}=0$ for $0\leq i \leq D$; cf. \cite[Theorem 3]{1988Ter} (or \cite[Theorem 4.2]{2001PascasioLAA}).
The antipodal $Q$-polynomial distance-regular graphs are completely classified for diameter $\geq 6$ and valency $\geq 3$ cases; cf. \cite{1996DicTer}.

\begin{lemma}[cf. {\cite[Corollary 11.5]{1999CaughmanIV}}]
Let $\Gamma$ be an antipodal $Q$-polynomial distance-regular graph with $D\geq 3$.
Fix a vertex $x$ of $\Gamma$ and write $T=T(x)$.
Let $W$ be an irreducible $T$-module with diameter $d$.
Recall the scalars $\{a_i(W)\}^d_{i=0}$ from \eqref{scalars:ai,wi}.
Then 
\begin{equation}\label{eq:a_i = a_(d-i)}
	a_i(W) = a_{d-i}(W) \qquad (0 \leq i \leq d).
\end{equation}
\end{lemma}
\begin{proof}
Apply the result of \cite[Corollary 11.5]{1999CaughmanIV} to the dual bipartite $Q$-polynomial distance-regular graphs.
\end{proof}


We give a remark.
Caughman \cite{1999CaughmanIV} showed that a \emph{bipartite} $Q$-polynomial distance-regular graph $\Gamma$ with $D\geq 3$ is thin and pseudo-vertex-transitive.
More specifically, he proved that any irreducible $T(x)$-module is both thin and dual thin \cite[Lemma 9.2]{1999CaughmanIV} and that the isomorphism class of an irreducible $T(x)$-module is completely determined by its endpoint and diameter \cite[Theorem 13.1]{1999CaughmanIV}. 
He also found a recurrence that gives the multiplicities of irreducible $T(x)$-modules; this recurrence is expressed as the intersection numbers, eigenvalues, and dual eigenvalues of $\Gamma$ \cite[Section 14]{1999CaughmanIV}. 
Thus the isomorphism class of the standard $T$-module of $\Gamma$ does not depend on the base vertex $x$.
By these comments, $\Gamma$ is thin and pseudo-vertex-transitive.
Dualizing the results of \cite{1999CaughmanIV}, we can find that the dual bipartite (or antipodal) $Q$-polynomial distance-regular graphs are thin and pseudo-vertex-transitive.
These comments are summarized as follows.
\begin{proposition}[{\cite{1999CaughmanIV}}]\label{prop:antiDRG pvt}
A bipartite and/or dual bipartite $Q$-polynomial distance-regular graph with $D\geq 3$ is thin and pseudo-vertex-transitive.
\end{proposition}

\noindent
In the next two sections, we will discuss pseudo-vertex transitivity of the antipodal $Q$-polynomial distance-regular graphs with diameter three and four. 
We can see that these graphs are pseudo-vertex-transitive by Proposition \ref{prop:antiDRG pvt}, but we will prove it in detail to explain how the thin property plays in determining their pseudo-vertex transitivity.

We now recall the tight distance-regular graphs.
Let $\Gamma$ denote a distance-regular graph with diameter $D\geq 3$ and eigenvalues $k=\theta_0 > \theta_1 > \cdots > \theta_D$.
In \cite{2000JKT}, Juri\v{s}i\'{c} \textit{et al.} proved that the intersection numbers $a_1, b_1$ satisfy
\begin{equation}\label{fund.equality}
	\left(\theta_1+\frac{k}{a_1+1}\right)\left(\theta_D+\frac{k}{a_1+1}\right) \geq - \frac{ka_1b_1}{(a_1+1)^2}.
\end{equation}
We say $\Gamma$ is \emph{tight} whenever $\Gamma$ is not bipartite and equality holds in \eqref{fund.equality}.
Define 
\begin{equation}\label{local eigval}
	b^+:=-1-\frac{b_1}{1+\theta_D}, \qquad b^-:=-1-\frac{b_1}{1+\theta_1}.
\end{equation}
We say $\Gamma$ is \emph{tight with respect to $x$} whenever every irreducible $T(x)$-module with endpoint $1$ is thin with local eigenvalue $b^+$ or $b^-$.
The tight distance-regular graphs are characterized by using the Terwilliger algebra \cite{2002GoTer}.
We state the following lemma.

\begin{lemma}[{\cite[Theorem 13.6]{2002GoTer}}, {\cite[Theorem 12.6]{2000JKT}}]\label{lem:tight char}
Let $\Gamma$ denote a distance-regular graph with diameter $D\geq 3$.
The following (i)--(iii) are equivalent.
\begin{enumerate}
	\item[(i)] $\Gamma$ is tight.
	\item[(ii)] For $x\in X$ the local graph of $\Gamma$ at $x$ is connected strongly regular with eigenvalues $a_1$,  $b^+$, $b^-$.
	\item[(iii)] $\Gamma$ is non-bipartite and tight with respect to each vertex.
\end{enumerate}
\end{lemma}

\noindent
Pascasio \cite{2001Pascasio} showed that the non-bipartite antipodal $Q$-polynomial distance-regular graphs are tight. 
For instance, the Johnson graph $J(2D,D)$, the halved $2D$-cube, the non-bipartite Taylor graphs, and the Meixner1 graph are non-bipartite antipodal $Q$-polynomial distance-regular, and therefore they are tight; see \cite{1996DicTer, 2000JKT}.

We finish this section with the following lemmas, which will be used later.

\begin{lemma}[cf. {\cite[Theorem 10.1]{2002GoTer}}]\label{lem:dual endpt W}
Let $\Gamma$ be a $Q$-polynomial distance-regular graph with $D\geq 3$.
Fix $x\in X$ and write $T=T(x)$ the Terwilliger algebra of $\Gamma$.
Let $W$ be a thin irreducible $T$-module with endpoint $1$. 
If $W$ has the local eigenvalue $b^+$ (resp. $b^-$), then $W$ has dual endpoint $1$ (resp. $2$).
\end{lemma}

\begin{lemma}[cf. {\cite[Theorem 10.7]{2002GoTer}}]\label{Thm10.7:TerGo}
Let $\Gamma$ be a $Q$-polynomial distance-regular graph with $D\geq 3$.
Fix $x\in X$ and write $T=T(x)$ and $E^*_i=E^*_i(x)$ for $0 \leq i \leq D$.
Let $W$ be a thin irreducible $T$-module with endpoint $1$ and local eigenvalue $\lambda \in \{b^+, b^-\}$.	
Then $W$ has dimension $D-1$. 
For $0\leq i \leq D$, $E^*_iW$ is zero if $i \in \{0,D\}$ and has dimension $1$ if $i \notin\{0,D\}$.
Moreover, $E_iW=0$ if $i \in \{0,n\}$ and has dimension $1$ if $i \notin\{0,n\}$, where $n=1$ if $\lambda=b^-$; and $n=D$ if $\lambda=b^+$.
\end{lemma}

\begin{lemma}[cf. {\cite[Theorem 11.1]{2002GoTer}}]\label{Thm11.7:TerGo}
Let $\Gamma$ be a $Q$-polynomial distance-regular graph with $D\geq 3$.
Fix a vertex $x$ of $\Gamma$ and write $T=T(x)$.
Let $W$ be a thin irreducible $T$-module with endpoint $1$ and local eigenvalue $\lambda \in \{b^+, b^-\}$.
Let $W'$ be an irreducible $T$-module. Then $W$ and $W'$ are isomorphic as $T$-modules if and only if $W'$ is thin with endpoint $1$ and local eigenvalue $\lambda$.
\end{lemma}


\section{Taylor graphs}\label{sec:D=3}
In this section, we discuss the pseudo-vertex transitivity of the Taylor graphs.
A \emph{Taylor graph} is a distance-regular antipodal double cover of a complete graph with diameter three, that is, a distance-regular graph with intersection array $\{k,b,1; 1,b,k\}$, where $b<k-1$.
We note that the non-bipartite Taylor graphs are exactly tight distance-regular graphs with diameter three \cite{2002JKDiscMath}.

For the rest of this section, we denote by $\Gamma$ a Taylor graph with vertex set $X$ and the eigenvalues $k=\theta_0>\theta_1>\theta_2>\theta_3$.
The spectrum of  $\Gamma$ is
\begin{equation*}
	\{k^1, \theta_1^f, \theta_2^k, \theta_3^g\},
\end{equation*}
where 
\begin{equation}\label{eigvals:Taylor}
	\theta_1 = \frac{k-2b-1+\sqrt{\mathcal{D}}}{2}, \qquad
	\theta_2=-1, \qquad 
	\theta_3 = \frac{k-2b-1-\sqrt{\mathcal{D}}}{2},
\end{equation}
and 
\begin{equation*}
	f= \left(\frac{1}{2}-\frac{k-2b-1}{2\sqrt{\mathcal{D}}}\right)(k+1), \qquad
	g=\left(\frac{1}{2}+\frac{k-2b-1}{2\sqrt{\mathcal{D}}}\right)(k+1),
\end{equation*}
where $\mathcal{D}=(k-2b-1)^2+4k$.
By Lemma \ref{lem:tight char}, for each $x\in X$ the local graph $\Delta=\Delta(x)$ of $\Gamma$ is connected strongly regular with parameters $(n_{\Delta}, k_{\Delta}, a_{\Delta}, c_{\Delta})$, where
\begin{equation*}
n_{\Delta} = k, \quad k_{\Delta}=k-b-1, \quad a_{\Delta}=\frac{2k-3b-4}{2}, \quad c_{\Delta}=\frac{k-b-1}{2}.
\end{equation*}
From \eqref{srg:eigval}, the three eigenvalues of $\Delta$ are given by
\begin{equation}\label{Taylor:local eig}
	k_\Delta, \qquad 
	\sigma = \frac{k-2b-3+\sqrt{\mathcal{D}}}{4},\qquad 
	\tau  = \frac{k-2b-3-\sqrt{\mathcal{D}}}{4},
\end{equation}
where $\mathcal{D} = (k-2b-1)^2+4k$.
One can verify that 
\begin{equation}\label{Taylor:b+,b-}
\sigma =  -1-\frac{b}{1+\theta_3}, \qquad \tau = -1-\frac{b}{1+\theta_1}.
\end{equation}
Thus, $\sigma=b^+$ and $\tau=b^-$, where $b^+$ and $b^-$ are from \eqref{local eigval}.
By \eqref{srg:mult} the multiplicities $m_\sigma$ and $m_\tau$ of $\sigma$ and $\tau$, respectively, are given by
\begin{equation}\label{Taylor:local mult}
	m_\sigma = \frac{k-1}{2} - \frac{(k+1)(k-2b-1)}{2\sqrt{\mathcal{D}}}, \qquad 
	m_\tau = \frac{k-1}{2} + \frac{(k+1)(k-2b-1)}{2\sqrt{\mathcal{D}}}.
\end{equation}
By these comments we have the following lemma.
\begin{lemma}\label{localspectrum}
For each $i\in\{1,2\}$ and  for all $x\in X$, the spectrum of $\Delta_i(x)$ of $\Gamma$ is determined by the intersection numbers $k$ and $b$.
\end{lemma}
\begin{proof}
For the case $i=1$, it follows from \eqref{Taylor:local eig} and \eqref{Taylor:local mult}.
For the case $i=2$, since $\Gamma$ is an antipodal double cover it follows $\Delta_1(x) \cong \Delta_1(\hat{x})=\Delta_2(x)$. 
The result follows.
\end{proof}


\begin{lemma}\label{lem:Taylor endpt=01}
Fix a vertex $x\in X$ and let $T=T(x)$ be the Terwilliger algebra of $\Gamma$.
Let $W$ be a $T$-module.
Then the endpoint of $W$ is either $0$ or $1$.
\end{lemma}
\begin{proof}
Suppose that $W$ has the endpoint $r$ and diameter $d$.
Clearly, the possible value for $r$ is $0, 1, 2$, or $3$.
If $r=2$, then $d=0$ or $d=1$.
If $d=0$, then $W$ is a one-dimensional $T$-module.
It implies that an eigenvalue of $E^*_2(x)AE^*_2(x)$ on $W$ is an eigenvalue of $\Gamma$.
By Lemma 7.1, the eigenvalue of $E^*_2(x)AE^*_2(x)$ is $\lambda \in \{\sigma, \tau\}$ in  \eqref{Taylor:b+,b-}.
But, $\lambda \notin \{k, \theta_1, \theta_2,\theta_3 \}$ in \eqref{eigvals:Taylor}, the eigenvalues of $\Gamma$.
If $d=1$, then $\dim(W)=2$ and $E^*_3(x)W \ne 0$.
Take $0 \ne w \in E^*_3(x)W=E^*_0(\hat{x})W \subseteq W$.
Since $W$ is a $T$-module, it follows $Mw \subseteq W$.
Observe that $Mw$ is the primary $T(\hat{x})$-module, and thus $\dim(Mw)=4$, a contradiction. 
If $r=3$, then diameter of $W$ must be $0$. 
This case is similar to the case of $r=2$ and $d=1$, so it is not possible.
The result follows.
\end{proof}


\begin{lemma}\label{lem:Taylor_thin}
Recall the eigenvalues $\theta_0>\theta_1>\theta_2>\theta_3$ of $\Gamma$.
Fix $x\in X$ and write $T=T(x)$.
Let $W$ be an irreducible $T$-module with endpoint $1$ and local eigenvalue $\lambda \in \{\sigma, \tau\}$, where $\sigma, \tau$ are from \eqref{Taylor:b+,b-}.
Then the following (i)--(iii) hold.
\begin{itemize}
	\item[(i)] $W$ has diameter $1$ and is thin.
	\item[(ii)] Let $w_0 \in E^*_1W$ be a nonzero vector. Set $w_1 = E^*_2Aw_0$. 
	Then $\{w_0, w_1\}$ is a basis for $W$.
	With respect to this basis, the matrix representing $A$ is given by
		\begin{equation}\label{action A on W (d=3)}
		\begin{pmatrix}
		\lambda & (\lambda-\theta_t)^2 \\
		1 & \lambda
		\end{pmatrix},
		\end{equation}
		where $t=1$ if $\lambda=\sigma$ and $t=2$ if $\lambda=\tau$.
	\item[(iii)] The local eigenvalues $\sigma, \tau$ of $W$ are given by 
	$\sigma = (\theta_1-\theta_2)/2$ and $\tau = (\theta_2-\theta_3)/2$.
\end{itemize}
\end{lemma}
\begin{proof}
(i): For the first assertion, it is similar to the proof of Lemma \ref{lem:Taylor endpt=01}.
For the second assertion, one knows that $\Gamma$ is tight. 
By Lemma \ref{lem:tight char}(iii), $W$ is thin.\\
(ii): Clearly, $\{w_0, w_1\}$ is an orthogonal basis for $W$.
By \eqref{eq:action Av_i}, the action of $A$ on $w_0$ is given by
$
	Aw_0 = a_0(W) w_0 + w_1.
$
Observe that $a_0(W)=\lambda$.
Next, the action of $A$ on $w_1$ is given by
$
	Aw_1 = x_1(W) w_0 + a_1(W) w_1,
$
where the scalars $a_1(W)$ and $x_1(W)$ are from \eqref{scalars:ai,wi}.
As $\Gamma$ is an antipodal double cover, by \eqref{eq:a_i = a_(d-i)} we find $a_1(W)=a_0(W)=\lambda$.
Recall the intersection numbers $b_0(W)$ and $c_1(W)$ of $W$.
Referring to \cite[Section 9]{2010Cerzo}, we find that $a_0(W)+b_0(W) = \theta_t$ and $c_1(W)+a_1(W)=\theta_t$, where $t$ is dual endpoint of $W$.
Using these two equations and the fact that $x_1(W)=b_0(W)c_1(W)$, we find $x_1(W)=(\lambda-\theta_t)^2$.
By Lemma \ref{lem:dual endpt W}, the dual endpoint of $W$ is $t=1$ (resp. $t=2$) if the local eigenvalue $\lambda$ is $\sigma$ (resp. $\tau$). 
By these comments, we obtain the matrix \eqref{action A on W (d=3)}.\\
(iii): By \eqref{sum:ai=thi}, we have $a_0(W)+a_1(W)=2\lambda = \theta_t+\theta_{t+1}$, where $t$ is the dual endpoint of $W$. By Lemma \ref{lem:dual endpt W}, the result follows.
\end{proof}

\begin{proposition}\label{Taylor:thin}
A Taylor graph is thin.
\end{proposition}
\begin{proof}
By Lemma \ref{lem:Taylor endpt=01} and Lemma \ref{lem:Taylor_thin}.
\end{proof}

We recall the following lemma, which classifies all irreducible $T$-modules of $\Gamma$ with endpoint one.

\begin{lemma}\label{lem:Taylor WisoW'}
Fix $x\in X$ and write $T=T(x)$.
Let $W$ be a thin irreducible $T$-module with endpoint $1$ and local eigenvalue $\lambda \in \{\sigma, \tau\}$, where $\sigma, \tau$ are from \eqref{Taylor:local eig}.
Let $W'$ denote an irreducible $T$-module. Then $W$ and $W'$ are isomorphic as $T$-modules if and only if $W'$ is thin with endpoint $1$ and local eigenvalue $\lambda$.
\end{lemma}
\begin{proof}
Apply Lemma \ref{Thm11.7:TerGo} to $\Gamma$. The result immediately follows.
\end{proof}

Recall the standard $T$-module $V$.
Since $T$ is semisimple and using  Lemma \ref{lem:Taylor endpt=01}, Lemma \ref{lem:Taylor_thin}, and Lemma \ref{lem:Taylor WisoW'}, $V$ decomposes into an orthogonal direct sum of thin irreducible $T$-modules 
\begin{equation}\label{eq:decompV D=3}
	V=W_0 \oplus (\bigoplus_i W_i) \oplus (\bigoplus_j W'_j),
\end{equation}
where we denote by
\begin{equation*}
\renewcommand{\arraystretch}{1.2}
\begin{tabular}{ccccc}
	\hline
	$T$-module  & dimension & endpoint & diameter & comment\\
	\hline
	$W_0$ & 4 & 0 & 3 & the primary module \\
	 $W_i$  & 2 & 1 & 1 &  local eigenvalue $\sigma$\\
	 $W'_j$  & 2 & 1 & 1 & local eigenvalue $\tau$ \\
	\hline
\end{tabular},
\end{equation*}
where $\sigma$ and $\tau$ are from \eqref{Taylor:b+,b-}.

\begin{theorem}\label{thm:D=3}
A Taylor graph is pseudo-vertex-transitive.
\end{theorem}
\begin{proof}
For a given vertex $x\in X$, let $T(x)$ be the Terwilliger algebra of $\Gamma$.
By Lemma \ref{localspectrum} and Lemma \ref{lem:Taylor WisoW'}, the isomorphism class of an irreducible $T(x)$-module is determined by the parameters $k$ and $b$.
Moreover, by Lemma \ref{lem:Taylor_thin} and \eqref{Taylor:local mult}, the multiplicities of irreducible $T(x)$-modules are determined by $k$ and $b$.
Therefore, the isomorphism class of the standard $T(x)$-module \eqref{eq:decompV D=3} of $\Gamma$ is completely determined by $k$ and $b$.
The result follows.
\end{proof}

\begin{remark}\label{rmk:Taylor graphs}
Taylor graphs are not vertex-transitive in general. 
For example, there are four different Taylor graphs with intersection array $\{25,12,1;1,12,25\}$ of which only one is vertex-transitive; cf. \cite{Spence}.
\end{remark}

We finish this section with a comment. 

\begin{proposition}
For each $x\in X$, the dimension of $T(x)$ of $\Gamma$ is given by
\begin{equation*}
	\dim(T(x))=24.
\end{equation*}
\end{proposition}
\begin{proof}
Applying Proposition \ref{Prop:Wedderburn} to \eqref{eq:decompV D=3}, the algebra $T(x)$ is isomorphic to 
\begin{equation*}
	\mathrm{Mat}_{4}(\mathbb{C}) \oplus
	\mathrm{Mat}_{2}(\mathbb{C}) \oplus 
	\mathrm{Mat}_{2}(\mathbb{C}).
\end{equation*}
From \eqref{T:dim formula}, it follows that $\dim(T(x))=4^2+2^2+2^2=24$.
\end{proof}

\section{Antipodal tight graphs $\mathrm{AT}4(p,q,2)$}\label{sec:D=4}

In this section, we consider a distance-regular antipodal double cover with diameter four.
We recall some properties of this graph that are needed in this section.
For more information, we refer the reader to \cite{2003Jurisic, 2000JKEuroJC, 2002JKDiscMath}.
Let $\Gamma$ denote a non-bipartite distance-regular antipodal double cover with diameter four. 
Then $\Gamma$ is $Q$-polynomial if and only if $\Gamma$ is tight \cite{2002JKDiscMath}.
Suppose that $\Gamma$ is $Q$-polynomial. 
The intersection array of $\Gamma$ is parameterized by $p$, $q$  (cf. \cite[Section 5]{2002JKDiscMath}):
\begin{equation}\label{int_array:AT4}
\{q(pq+p+q), (q^2-1)(p+1), q(p+q)/2, 1; 1, q(p+q)/2, (q^2-1)(p+1); q(pq+p+q)\},
\end{equation}
where $p\geq 1$, $q\geq 2$ are integers.
The spectrum of $\Gamma$ is given by $\{\theta_0^{m_0}, \theta_1^{m_1}, \theta_2^{m_2}, \theta_3^{m_3}, \theta_4^{m_4} \}$, where 
\begin{equation}\label{eigvals:AT4(p,q,2)}
	\theta_0 = q(pq+p+q), \qquad \theta_1=pq+p+q, \qquad
	\theta_2 = p, \qquad \theta_3=-q, \qquad \theta_4=-q^2,
\end{equation}
and $m_0=1$,
\begin{align*}
	&m_1  = \frac{q(pq^2+q^2+pq-p)}{p+q}, &&
	m_2  = \frac{q(pq+p+q)(q^2-1)(2q+pq+p)}{(p+q)(p+q^2)},\\
	& m_3  = \frac{(pq^2+q^2+pq-p)(pq+p+q)}{p+q}, &&
	m_4  = \frac{(p+1)(pq+p+q)(pq^2+q^2+pq-p)}{(p+q)(p+q^2)}.
\end{align*}
For each $x\in X$ the local graph $\Delta=\Delta(x)$ of $\Gamma$ is connected strongly regular with parameters $(n_\Delta, k_\Delta, a_\Delta, c_\Delta)$, where (cf. \cite[Section 5]{2002JKDiscMath})
\begin{equation}\label{AT4(p,q,2)local para}
	n_\Delta=q(pq+p+q), \qquad k_\Delta=p(q+1), \qquad a_\Delta=2p-q, \qquad c_\Delta=p.
\end{equation}
The nontrivial eigenvalues of $\Delta$ are $b^+=p$ and $b^-=-q$ and their multiplicities are
\begin{equation}\label{AT4(p,q,2);mult}
	m_{b^+} = \frac{(q^2-1)(pq+p+q)}{p+q}, \qquad m_{b^-} = \frac{pq(q+1)(p+1)}{p+q}.
\end{equation}
Let $\mathrm{AT4}(p,q,r)$ denote an antipodal tight $r$-cover of diameter four with parameters $p$ and $q$.
Clearly, $\Gamma$ is the same thing as  $\mathrm{AT4}(p,q,2)$.
There are three known examples of an $\mathrm{AT4}(p,q,2)$, namely the following graphs with array $\{c_1,c_2,c_3,c_4\}$:
\begin{itemize}[itemsep=0.5pt]
	\item[(i)] $\mathrm{AT4}(2,2,2)$, the Johnson graph $J(8,4)$ with $\{1,4,9,16\}$;
	\item[(ii)] $\mathrm{AT4}(4,2,2)$, the half-cube $\frac{1}{2}H(8,2)$ with $\{1,6,15,28\}$;
	\item[(iii)] $\mathrm{AT4}(8,4,2)$, the Meixner1 graph with $\{1,24,135,176\}$.
\end{itemize}
\begin{remark}
In \cite[Theorem 1.1]{1996DicTer}, Dickie and Terwilliger gave a family of graphs with array $\{1, \beta\eta, (\beta^2-1)(2\eta-\beta+1),\beta(2\eta+2\eta\beta-\beta^2)\}$, where $\beta\geq 2$, $\eta \geq 3\beta/4$ are integers and $\eta$ divides $\beta^2(\beta^2-1)/2$. 
We remark that $\beta\eta$ is an even integer by \cite[Corollary 3.2]{2000JKEuroJC}.
With respect to the graph $\mathrm{AT}4(p,q,2)$, we note that $p=2\eta-\beta$ and $q=\beta$.
Using this and by \eqref{eigvals:AT4(p,q,2)}, the eigenvalues of $\mathrm{AT4}(p,q,2)$ are given in terms of $\beta$ and $\eta$ as follows. 
$$
	\theta_0 = \beta(2\eta+2n\beta-\beta^2), \quad
	\theta_1 = 2\eta+2n\beta-\beta^2, \quad
	\theta_2 = 2\eta, \quad
	\theta_3 = -\beta, \quad
	\theta_4 = -\beta^2.
$$
If $(\beta,\eta)=(2,2)$ we have $\mathrm{AT}4(2,2,2)$, and if $(\beta,\eta)=(2,3)$ we have $\mathrm{AT}4(4,2,2)$; these graphs are all examples known for $\beta=2$.
In addition, if $(\beta,\eta)=(4,6)$ we have the Meixner1 graph $\mathrm{AT}4(8,4,2)$, which is the only known example for $\beta>2$ so far.
\end{remark}

For the rest of this section, we denote by $\Gamma$ an antipodal tight graph $\mathrm{AT4}(p,q,2)$ with vertex set $X$.
Fix $x\in X$ and write $T=T(x)$ the Terwilliger algebra of $\Gamma$ and $E^*_i=E^*_i(x)$ the $i$th dual primitive idempotent of $\Gamma$ for $0 \leq i \leq 4$.
We recall a classification of all thin irreducible $T$-modules of $\Gamma$ with endpoint $1$.

\begin{lemma}\label{lem:W endpt=1}
Let $W$ denote a thin irreducible $T$-module with endpoint $1$ and local eigenvalue $\lambda \in \{p, -q\}$.
Then $W$ has diameter $2$.
If $W'$ is an irreducible $T$-module, then $W$ and $W'$ are isomorphic as $T$-modules if and only if $W'$ is thin with endpoint $1$ and local eigenvalue $\lambda$.
\end{lemma}
\begin{proof}
Apply Lemma \ref{Thm10.7:TerGo} and Lemma \ref{Thm11.7:TerGo}.
\end{proof}

Next, we discuss irreducible $T$-modules of $\Gamma$ with endpoint $2$.

\begin{lemma}\label{lem:W endpt=2}
If $W$ is an irreducible $T$-module with endpoint $2$, then $W$ has diameter $0$ and is thin.
\end{lemma}
\begin{proof}
Let $d$ denote the diameter of $W$.
Observe that the possible value for  $d$ is $0$, $1$, or $2$.
If $d=2$, then $\dim(W)=3$ and $E^*_4W=E^*_0(\hat{x})W\ne 0$.
Take a nozero vector $w \in E^*_0(\hat{x})W$.
Since $W$ is a $T$-module, we have $Mw \subseteq W$. 
Observe that $Mw$ is the primary $T(\hat{x})$-module. 
It follows $\dim (Mw)=5$, which is greater than $\dim(W)=3$, a contradiction.
If $d=1$, then $E^*_3W=E^*_1(\hat{x})W$ is nonzero. 
It follows that $W$ is an irreducible $T(\hat{x})$-module with endpoint $1$.
As $\Gamma$ is tight, by Lemma \ref{lem:tight char}(iii) $W$ is thin and has local eigenvalue $\lambda\in \{p, -q\}$.
By Lemma \ref{lem:W endpt=1}, we have $d=2$, a contradiction.
By these comments, we have $d=0$.
Since $W$ is an irreducible $T$-module, the dimension of $W$ is one. The result follows.
\end{proof}

\begin{lemma}\label{lem:W endpt=2 classf}
Let $W$ be an irreducible $T$-module with endpoint $2$.
Let $\eta$ be an eigenvalue of $E^*_2AE^*_2$ with an eigenspace $W$.
If $W'$ is another irreducible $T$-module with endpoint $2$, then $W$ and $W'$ are isomorphic as $T$-modules if and only if $W'$ is an eigenspace associated with the eigenvalue $\eta$.
\end{lemma}
\begin{proof}
Let $\eta'$ denote an eigenvalue of $E^*_2AE^*_2$ with the eigenspace $W'$.
By Lemma \ref{lem:W endpt=2}, $W$ and $W'$ are both one-dimensional. 
Let $\rho$ denote a vector space isomorphism from $W$ and $W'$.
Apparently, $(E^*_i \rho - \rho E^*_i)W=0$ for all $0 \leq i \leq 4$.
Also, we find that $(A\rho-\rho A)W=0$  if and only if $\eta=\eta'$.
The result follows.
\end{proof}

\begin{lemma}\label{lem:q4,p,q}
Recall $\Gamma$ an antipodal tight graph $\mathrm{AT4}(p,q,2)$.
For each vertex $x\in X$ and for each $1 \leq i \leq 4$, the spectrum of $i$th subconstituent $\Delta_i(x)$ of $\Gamma$  is determined by the parameters $p, q$. 
\end{lemma}
\begin{proof}
Fix a vertex $x$ in $\Gamma$.
Since $\Gamma$ is tight, $\Delta_1(x)=\Delta_3(x)$ is strongly regular with the parameters 
$
	(q(pq+p+q), p(q+1), 2p-q,p),
$
and has the nontrivial eigenvalues $p, -q$.
It suffices to show that the spectrum of the second subconstituent $\Delta_2=\Delta_2(x)$ of $\Gamma$ is determined by $p$ and $q$.
We first assume the Krein parameter $q^4_{44}=0$.
By \cite[Theorem 5.5]{2003Jurisic}, $\Delta_2$ is an antipodal distance-regular graph with diameter four, and intersection array and the spectrum of $\Delta_2$ are determined by $p$ and $q$. 
Now we assume $q^4_{44}\ne 0$.
As the graph $\Gamma$ is $1$-homogeneous in the sense of Nomura (cf. \cite[Theorem 11.7]{2000JKT}), $\Delta_2$ is edge-regular. 
We claim that (i) $\Delta_2$ has at most $7$ distinct eigenvalues; and (ii) each eigenvalue and its multiplicity are determined by $p$ and $q$.

First, we find that $a_2$ is the eigenvalue of $\Delta_2$ with multiplicity $1$.
Next, let $W$ denote a thin irreducible $T$-module with endpoint $1$ and local eigenvalue $\lambda \in \{p, -q\}$.
By Lemma \ref{lem:W endpt=1}, $W$ has diameter $2$.
Take a nonzero vector $w_0 \in E^*_1W$. 
Set $w_i = E^*_{i}Aw_{i-1}$ for $i=1,2$.
Then $\{w_i\}^2_{i=0}$ is a basis for $W$, and from \eqref{eq:action Av_i} the action of $A$ on the basis $\{w_i\}^2_{i=0}$ is given by
$$
	Aw_i = w_{i+1} + a_i(W)w_i + x_i(W)w_{i-1}. 
$$
Observe that $a_0(W)=\lambda$ and $a_0(W)=a_2(W) =\lambda$ by \eqref{eq:a_i = a_(d-i)}. 
Observe also that $a_1(W)$ is an eigenvalue of $E^*_2AE^*_2$, i.e., an eigenvalue of $\Delta_2$. 
By this and using \eqref{sum:ai=thi} and Lemma \ref{lem:dual endpt W}, we have
$$
	a_1(W) = \theta_t + \theta_{t+1} + \theta_{t+2} - 2\lambda,
$$
where $t=1$ if $\lambda=p$ and $t=2$ if $\lambda=-q$.
From this and by \eqref{eigvals:AT4(p,q,2)}, we find that $a_1(W)$ is determined by $p, q$.
Since each $\lambda$ yields the scalar $a_1(W)$, multiplicities of $\lambda$ and $a_1(W)$ are equal to each other.
By Lemma \ref{lem:W endpt=1} there are two non-isomorphic $T$-modules with endpoint $1$.
It follows that there are two eigenvalues of $\Delta_2$, which are obtained from irreducible $T$-modules with endpoint $1$.

Finally, let $W$ denote a thin irreducible $T$-module with endpoint $2$.
By Lemma \ref{lem:W endpt=2}, $W$ is a one-dimensional eigenspace for $E^*_2AE^*_2$.
It follows that the corresponding eigenvalue of $\Delta_2$ becomes an eigenvalue of $\Gamma$. 
From this, we find that $\Delta_2$ has four possible eigenvalues $\theta_1$, $\theta_2$, $\theta_3$, $\theta_4$. 
For each $1 \leq i \leq 4$, let $m_i$ denote the multiplicity of $\theta_i$.
Abbreviate $B=E^*_2AE^*_2$.
Since $\Delta_2$ is edge-regular and $\mathrm{trace} (B^\ell) = \sum^4_{i=1} \theta^\ell_i m_i$, $\ell=0,1,2,3$, we can solve this system of equations for $m_i$, $1\leq i \leq 4$, where $m_i$ are determined by $p, q$.
The result follows.
\end{proof}
\noindent
We remark that the proof of Lemma \ref{lem:q4,p,q} is motivated to give a new feasibility condition for an $\mathrm{AT}4(p,q,r)$; see \cite{2022XLK}

\begin{proposition}\label{prop:AT4thin}
An antipodal tight graph $\mathrm{AT}4(p,q,2)$ is thin.
\end{proposition}
\begin{proof}
For any vertex $x$, an irreducible $T(x)$-module of $\mathrm{AT}4(p,q,2)$ with endpoint $1$ is thin since $\mathrm{AT}4(p,q,2)$ is tight.
In addition, every irreducible $T(x)$-module of $\mathrm{AT}4(p,q,2)$ with endpoint $2$ is thin by Lemma \ref{lem:W endpt=2}.
The result follows.
\end{proof}

Recall the standard $T$-module $V$.
Since $T$ is semisimple and using Lemma \ref{lem:W endpt=1} and Lemma \ref{lem:W endpt=2 classf}, $V$ decomposes into an orthogonal direct sum of thin irreducible $T$-modules 
\begin{equation}\label{eq:decompV D=4}
	V=W_0 \oplus (\bigoplus_i W_i) \oplus (\bigoplus_j W'_j) \oplus (\bigoplus_h W''_h),
\end{equation}
where we denote by
\begin{equation*}
\renewcommand{\arraystretch}{1.2}
\begin{tabular}{ccccc}
	\hline
	$T$-module  & dimension & endpoint & diameter & comment\\
	\hline
	$W_0$ & 5 & 0 & 4 & the primary module \\
	 $W_i$  & 3 & 1 & 2 &  local eigenvalue $p$\\
	 $W'_j$  & 3 & 1 & 2 & local eigenvalue $-q$ \\
	 $W''_h$ & 1 & 2  & 0 & eigenspace of $E^*_2AE^*_2$\\
	\hline
\end{tabular}.
\end{equation*}

\begin{theorem}\label{thm:D=4}
An antipodal tight graph $\mathrm{AT}4(p,q,2)$ is pseudo-vertex-transitive.
\end{theorem}
\begin{proof}
For a given vertex $x\in X$, let $T(x)$ be the Terwilliger algebra of $\mathrm{AT}4(p,q,2)$. 
By Lemma \ref{lem:W endpt=1} and Lemma \ref{lem:W endpt=2 classf}, the isomorphism class of an irreducible $T(x)$-module is determined by the parameters $p$ and $q$.
Moreover, by Lemma \ref{lem:q4,p,q} together with \eqref{AT4(p,q,2);mult}, the multiplicities of irreducible $T(x)$-modules are determined by $p$ and $q$.
Therefore, the isomorphism class of the standard $T(x)$-module \eqref{eq:decompV D=4} is completely determined by $p$ and $q$. The result follows.
\end{proof}

We finish this section with a comment. 

\begin{proposition}
For each $x\in X$ the dimension of $T(x)$ of $\mathrm{AT}4(p,q,2)$ is given by
\begin{equation*}
	\dim(T(x))=\ell+43,
\end{equation*}
where  $\ell (\leq 7)$ is the number of distinct eigenvalues of $\Delta_2(x)$ except for the  eigenvalue $a_2$.
\end{proposition}
\begin{proof}

Applying Proposition \ref{Prop:Wedderburn} to \eqref{eq:decompV D=4}, the  algebra $T(x)$ is isomorphic to 
\begin{equation*}
	\mathrm{Mat}_{5}(\mathbb{C}) \oplus
	\mathrm{Mat}_{3}(\mathbb{C}) \oplus 
	\mathrm{Mat}_{3}(\mathbb{C}) \oplus 
	\big(\mathrm{Mat}_{1}(\mathbb{C})^{\oplus \ell}\big),
\end{equation*}
where $\ell$ is the number of distinct eigenvalues of $\Delta_2(x)$ except for the  eigenvalue $a_2$. 
From \eqref{T:dim formula}, it follows that 
$$
\dim(T(x))=\ell(1)^2+ 3^2+3^2+5^2=  \ell+43.
$$
\end{proof}

\section{Concluding remarks}\label{Sec:conclusion}
In this paper, we discussed pseudo-vertex transitivity of $Q$-polynomial distance-regular graphs with small diameter $D\in \{2,3,4\}$.
We briefly summarize the results.
For the case $D=2$, we showed that the local spectrum of a strongly regular graph characterizes its pseudo-vertex transitivity.
In addition, we showed that two strongly regular graphs are $T$-isomorphic if and only if they have the same parameters and their local spectra are equal to each other.
For the case $D=3$, we were concerned with a Taylor graph $\Gamma$ with intersection array $\{k,b,1;1,b,k\}$.
We showed that $\Gamma$ is thin and that the isomorphism classes of the standard module of the Terwilliger algebra $T(x)$ of $\Gamma$ only depend on the parameters $k,b$, and hence $\Gamma$ is pseudo-vertex-transitive.
For the case $D=4$, we were concerned with an antipodal tight graph $\mathrm{AT}4(p,q,2)$.
We proved that $\mathrm{AT}4(p,q,2)$ is thin and that the isomorphism classes of the standard module of Terwilliger algebra $T(x)$ of $\mathrm{AT}4(p,q,2)$ only depend on the parameters $p,q$, and hence $\mathrm{AT}4(p,q,2)$ is pseudo-vertex-transitive.
In both cases, $D=3$ and $D=4$, the thin property of distance-regular antipodal double covers played an important role in determining their pseudo-vertex transitivity.
We also noted that bipartite and/or dual bipartite $Q$-polynomial distance-regular graphs are thin and pseudo-vertex-transitive.
We anticipate that the thinness of a distance-regular graph still plays a significant role in determining pseudo-vertex transitivity for $D\geq 3$.
We give the following conjecture.
\begin{conjecture}
A thin $Q$-polynomial distance-regular graph with diameter $D\geq 3$ is pseudo-vertex-transitive.
\end{conjecture}
As we saw in Problem \ref{IK problem}, the Ito-Koolen problem asks to classify all thin pseudo-vertex-transitive $Q$-polynomial distance-regular graphs with large enough diameter.
As a first step towards solving the Ito-Koolen problem, Tan \textit{et al.} \cite{2020TKCP} found an upper bound of the intersection number $c_2$ and parameter $\alpha$, respectively, for a thin $Q$-polynomial distance-regular graph with classical parameters $(D, b, \alpha, \beta)$. 
Each upper bound is given by a function in $b$, where $b=b_1/(\theta_1+1)$ and where $\theta_1$ is the second largest eigenvalue of $\Gamma$.
We give some suggestions for further research.
\begin{problem}
With reference to the result of \cite{2020TKCP}, improve the upper bound on $c_2$.
\end{problem}
\begin{problem}
Assume $\Gamma$ has the same parameters as a Grassmann graph $J_q(2D+1, D)$ and that $\Gamma$ is thin. 
Prove or disprove that $\Gamma$ and $J_q(2D+1, D)$ are $T$-isomorphic.
We remark that Liang \textit{et al.} \cite{TKL} have proved this problem for large enough diameter $D$.
\end{problem}

\section*{Acknowledgements}
The authors would like to express many thanks to the anonymous referees for their corrections and valuable comments.
J.-H. Lee would like to thank Paul Terwilliger for several valuable conversations and comments.
J.H. Koolen is partially supported by the National Key R and D Program of China (No. 2020YFA0713100), the National Natural Science Foundation of China (No. 12071454), and the Anhui Initiative in Quantum Information Technologies (No. AHY150000). 
Y.-Y Tan is supported by the National Natural Science Foundation of China (No.~11801007,~12171002) and Natural Science Foundation of Anhui Province (No. 1808085MA17) and the  foundation of Anhui Jianzhu University (No. 2018QD22).

\end{document}